\documentclass[a4paper,12pt]{article}

\usepackage[normalem]{ulem}
\usepackage{amsthm}                     
\usepackage[all]{xy}    
\usepackage{xcolor}
\usepackage{amsfonts,amssymb,amsmath}
\usepackage[pdftex]{hyperref}
\usepackage[english]{babel}                
\usepackage{authblk}                 

\usepackage[left=2.5cm, right=2.5cm, top=3cm, bottom=3cm]{geometry}

\newcommand{\R}{\mathbb{R}}
\newcommand{\C}{\mathbb{C}}
\newcommand{\Z}{\mathbb{Z}}
\newcommand{\cF}{\mathcal{F}}
\newcommand{\cW}{\mathcal{W}}
\newcommand{\bbA}{\mathbb{A}}
\newcommand{\bbT}{\mathbb{T}}

\newcommand{\ihom}{\underline{\hom}}
\newcommand{\dg}{{\mathrm{dg}} }

\newcommand{\bI}{\mathbf{I}}
\newcommand{\bC}{\mathbf{C}}
\newcommand{\bE}{\mathbf{E}}
\newcommand{\bM}{\mathbf{M}}
\newcommand{\bD}{\mathbf{D}}

\renewcommand{\Re}{\operatorname{Re}}
\renewcommand{\Im}{\operatorname{Im}}

\newcommand{\absHodge}{\mathrm{abs.~Hodge}}
\newcommand{\Hodge}{\mathbf{H}}

\DeclareMathOperator{\Gr}{Gr}
\DeclareMathOperator{\id}{id}

\DeclareMathOperator*{\colim}{colim}
\DeclareMathOperator{\Ind}{Ind}

\DeclareMathOperator{\reg}{reg}

\DeclareMathOperator{\map}{map}
\DeclareMathOperator{\Map}{Map}

 \newcommand{\Aut}{\mathrm{Aut}}
\newcommand{\xto}{\xrightarrow}
\newcommand{\Fun}{{\mathbf{Fun}}}

\DeclareMathOperator{\Sm}{\mathbf{Sm}}

\newcommand{\spa}{\mathbf{Sp}_{S^{1}}^{mot}}
\newcommand{\Nerve}{{\tt N}}
\renewcommand{\P}{{\mathbb{P}}}
\newcommand{\Ab}{{\mathbf{Ab}}}

\newcommand{\nat}{{\mathbb{N}}}

\DeclareMathOperator{\Hom}{Hom}

\renewcommand{\lim}{{\tt lim}}

\newcommand{\Sp}{\mathbf{Sp}}	
\newcommand{\Mf}{\mathbf{Mf}}
\newcommand{\Ch}{{\mathbf{Ch}}}

\newcommand{\sSet}{{\mathbf{sSet}}}

\newcommand{\Mod}{{\mathbf{Mod}}}
\newcommand{\MotSp}{\mathbf{Sp}^{\mathbb{P}^1}}
\newcommand{\Spc}{\mathbf{Spc}}
\newcommand{\MotSpc}{\mathbf{Spc}^{{mot}}}
\newcommand{\Cat}{\mathbf{Cat}}

\theoremstyle{plain}
\newtheorem{thm}{Theorem}[section]
\newtheorem{corollary}[thm]{Corollary}
\newtheorem{proposition}[thm]{Proposition}
\newtheorem{lemma}[thm]{Lemma}

\theoremstyle{definition}
\newtheorem{definition}[thm]{Definition}

\newtheorem{construction}[thm]{Construction}

\theoremstyle{remark}
\newtheorem{remark}[thm]{Remark}

\newtheorem{example}[thm]{Example}

\newcommand{\Vect}{\mathcal{V}ect}

\newcommand{\Spec}{\mathrm{Spec}}
\newcommand{\MHC}{\mathbf{MHC}}
\newcommand{\MHS}{\mathbf{MHS}}
\newcommand{\MHK}{\mathbf{MHK}}
\newcommand{\N}{\mathrm{N}}

\newcommand{\A}{\mathbb{A}}
\newcommand{\SH}{\mathbf{SH}}
\newcommand{\bH}{\mathbf{H}}
\newcommand{\bA}{\mathbf{A}}
\newcommand{\bK}{\mathbf{K}}
\newcommand{\Rig}{\mathbf{Rig}}
\newcommand{\Ring}{\mathbf{Ring}}
\newcommand{\CommGrp}{\mathbf{CGrp}}
\newcommand{\G}{\mathbb{G}}
\newcommand{\CAlg}{\mathbf{CAlg}}

\newcommand{\wA}{\widetilde{A}}
\newcommand{\iso}{\xrightarrow{~\sim~}}
\newcommand{\PR}{\mathbf{Pr}}

\DeclareMathOperator{\Cone}{Cone}
\DeclareMathOperator{\coker}{coker}
\DeclareMathOperator{\Ho}{Ho}
\newcommand{\IDR}{\mathbf{IDR}}

\newcommand{\calK}{\mathcal{K}}

\newcommand{\Symm}{\mathbf{Symm}}

\title{The Beilinson regulator is a map of ring spectra}
\author{%
Ulrich Bunke\thanks{Fakult\"at f\"ur Mathematik, Universit\"at Regensburg, 93040 Regensburg, Germany},
Thomas Nikolaus\thanks{FB Mathematik und Informatik, Universit\"at M\"unster, Einsteinstr. 62, 48149 M\"unster, Germany},
Georg Tamme\thanks{Fakult\"at f\"ur Mathematik, Universit\"at Regensburg, 93040 Regensburg, Germany}
\thanks{The authors are supported by the SFB/CRC 1085 \emph{Higher Invariants} (Universit\"at Regensburg) funded by the DFG.}
}

\begin{document}

\maketitle

\begin{abstract}
We prove that the Beilinson regulator, which is a map from $K$-theory to absolute Hodge cohomology of a smooth variety, admits a refinement to a map of $E_\infty$-ring spectra in the sense of algebraic topology. 
To this end we exhibit absolute Hodge cohomology as 
the cohomology of a commutative differential graded algebra over $\R$. The associated spectrum to this CDGA is the target of the refinement of the regulator and the usual $K$-theory spectrum is the source. 
To prove this result we compute the space of maps from the motivic $K$-theory spectrum to the motivic spectrum that represents absolute Hodge cohomology using the motivic Snaith theorem. We identify those maps which admit an $E_\infty$-refinement and   prove a uniqueness result for these refinements.
\end{abstract}

\section{Introduction}

Let $X$ be a smooth algebraic variety over {$\C$}.  
On the one hand,  we can form its algebraic $K$-groups $K_{*}(X)$,
which encode  information about the symmetric monoidal category of vector bundles on $X$ with respect to the direct sum. 
On the other hand, we have the Betti cohomology groups $H^{*}(X(\C), \R)$, which carry a natural mixed $\R$-Hodge structure by  \cite{Deligne-HodgeII}. In \cite{BeilinsonHodge} Beilinson constructs a natural complex in the derived category of mixed $\R$-Hodge structures whose cohomology is the Betti cohomology with its mixed Hodge structure, and he defines absolute cohomology groups $H^{*}_{\absHodge}(X,\R(i))$ of $X$ as Ext-groups of the Tate Hodge structure $\R(-i)$ and this complex. The absolute Hodge cohomology groups are the target of the
Beilinson regulator, a natural homomorphism of graded groups
\begin{equation}\label{reg}
\reg:  K_{{*}}(X) \to \bigoplus_{i \in \mathbb{N}}  H^{{2i}-*}_{\absHodge}(X,\R(i))\ .
\end{equation}

The tensor product of vector bundles induces a commutative ring structure on $K_{{*}}(X)$, and  the $\cup$-product in the Betti cohomology of $X$ provides a commutative  ring structure on $ \bigoplus_{i \in \mathbb{N}}  H^{*}_{\absHodge}(X,\R(i))$. 
It is known that  the regulator is a homomorphism of rings  {\cite[2.35]{GilletRR}}.

Our main motivation for the present paper is the application of the regulator to a multiplicative version of differential algebraic $K$-theory as discussed in the series of papers \cite{bg, buta, buta2}. For this, one needs a more refined version of the regulator map: 
The algebraic $K$-groups $K_{*}(X)$  are defined as  homotopy groups of an algebraic $K$-theory  spectrum $\calK(X)$ and  the multiplication on the $K$-groups is induced by an $E_\infty$-ring structure
  on this spectrum.
Absolute Hodge cohomology on the other hand is defined as Ext-groups in the derived category of mixed Hodge complexes. We realize the absolute Hodge cohomology groups as the cohomology groups of a  specific  chain complex 
$\IDR(X)$ {(Definition \ref{dqqlkwdqwdwqdwqd})} consisting of differential forms. The usual wedge product of forms gives a multiplication on the chain level, i.e.~it makes $\IDR(X)$ into a commutative differential graded algebra. 
Under the Eilenberg-MacLane equivalence $H$ this commutative differential graded algebra induces   an $E_\infty$-ring spectrum $H(\IDR(X))$
whose homotopy groups are the cohomology groups of $\IDR(X)$ and therefore the absolute Hodge cohomology groups of $X$. 
For the application we have in mind, it is an important question 
 whether the regulator \eqref{reg} can 
be refined to a spectrum map that is compatible with the $E_\infty$-ring structures on this level.

\begin{thm}\label{firstthm}
The regulator admits a refinement to a map of $E_\infty$-ring spectra 
$$
\calK(X) \to H(\IDR(X)) 
$$
which is natural in the variety $X$.
\end{thm}

In order to understand the refinement of the regulator and its construction, we employ techniques of motivic homotopy theory or, more precisely, motivic spectra. The assignment $X \mapsto \calK(X)$ is itself represented by such a motivic spectrum 
$\bK$. In detail, we have 
$$
\calK(X) \simeq \map(\Sigma^\infty_{+} X, \bK)
\ , \quad K_{n}(X) \cong \pi_n(\map(\Sigma^\infty_{+} X, \bK))\ .$$ 
The motivic spectrum {$\bK$  is a motivic $E_\infty$-ring spectrum, and this structure induces the $E_{\infty}$-ring spectrum structure on $\calK(X)$ under the above equivalence.}

Our next step towards proving that the Beilinson regulator is a morphism
 of $E_{\infty}$-ring spectra is to show that absolute Hodge cohomology can be represented by a motivic $E_\infty$-ring spectrum  {as well}.
\begin{thm}\label{thm:Hodge-Intro}
{Absolute} Hodge cohomology is representable by a motivic $E_\infty$-ring spectrum $\Hodge$.   Its underlying  {motivic} spectrum decomposes as ${{\bigoplus}}_{i} \Hodge(i)$ and we have 
$$
H^{2i-n}_{\absHodge}(X,\R(i)) \cong \pi_n( \map(\Sigma^\infty_{+} X,\Hodge(i)))
$$
\end{thm}
This may be seen as a refinement of \cite[Section 3]{Holmstrom-Scholbach} where a ring in the homotopy category is constructed which represents weak absolute Hodge cohomology, a version of absolute Hodge cohomology disregarding the weight filtrations. 
{A related result using different techniques is \cite[Prop.~1.4.10]{Deglise-Mazzari}.}
The motivic spectrum $\Hodge$ is constructed in {Section}~\ref{sakbcsakjcascsacac8789} {in} such a way that we have  equivalences of $E_\infty$-ring spectra 
$$H(\IDR(X)) \simeq \map(\Sigma^\infty_{+} X, \Hodge)\ .$$ 
After this translation in the motivic world, our main result is implied by the following statement which is the key technical fact:
\begin{thm}\label{thm:main-thm-intro}
There is  a  morphism of motivic spectra \begin{equation}\label{weklfjwekljwelkop23iopr2r232r}
 \bK \to \Hodge
 \end{equation}
which induces the Beilinson regulator $\reg$ as in \eqref{reg} {on the represented cohomology theories}. {There is a choice of this morphism which induces a   map of  rings in the homotopy category of motivic spectra. Furthermore, any morphism \eqref{weklfjwekljwelkop23iopr2r232r} which  induces a ring map   in the homotopy category of motivic spectra refines essentially uniquely to  a morphism of $E_{\infty}$-ring spectra.}
\end{thm}
{In this statement   `essentially unique' has to be interpreted in the appropriate higher categorical sense, i.e.~the moduli space of 
all such refinements is contractible.
Note that the uniqueness assertion in the last statement  of Theorem \ref{thm:main-thm-intro} provides a distinguished   refinement, up to contractible choice,   whose existence is asserted in Theorem \ref{firstthm}}. 
 We consider these uniqueness assertions as one of the main results.

\paragraph{Notations and conventions}

The current paper is written in the language of $\infty$-categories. We have collected some facts that we use in the Appendix. Apart from that we follow mostly Lurie in his terminology as in \cite{HTT} and \cite{HA}. 
We write $\CAlg(\bC)$ for the $\infty$-category of  commutative algebra objects in a symmetric monoidal $\infty$-category $\bC$.
In particular, we  write $\CAlg(\Sp)$ for the $\infty$-category of $E_\infty$-ring spectra and we will not use the $E_\infty$-terminology that we have used in the introduction. 

\paragraph{Acknowledgements} We thank Markus Spitzweck and David Gepner for several valuable hints and discussions.

\tableofcontents

\section{Mixed Hodge complexes}

In this section we recall the notion of mixed $\R$-Hodge structures and Beilinson's description of the derived category of mixed $\R$-Hodge structures in terms of the more flexible mixed $\R$-Hodge complexes.
We enhance this to an equivalence of stable $\infty$-categories. These $\infty$-categories are symmetric monoidal, and the main result is the construction of an explicit symmetric monoidal dg-model for them.
Basic references for (mixed) Hodge structures and complexes are \cite{Deligne-HodgeII,BeilinsonHodge}, or the book \cite{Peters-Steenbrink}.

In the following, all filtrations are assumed to be separated, exhaustive, and of finite length.

\begin{definition}[{see \cite[Prop.~2.1.9]{Deligne-HodgeII}}]
A \emph{pure $\R$-Hodge structure of weight} $n\in\Z$ is a pair $(H, \cF)$ consisting of a finite
dimensional $\R$-vector space $H$ and a decreasing filtration $\cF$ on
$H_\C:= H\otimes_{\R} \C$ satisfying
\[
H_\C \cong \cF^pH_\C \oplus \overline{\cF^{n-p+1}H_\C}
\]
for all $p\in \Z$, where $(\overline{-})$ denotes complex conjugation.
\end{definition}

\begin{definition}[{see \cite[Def.~2.3.1]{Deligne-HodgeII}}]
	\label{defmix}
A \emph{mixed $\R$-Hodge structure} is a triple $(H, \cW, \cF)$ consisting of a finite dimensional
$\R$-vector space $H$, an increasing filtration $\cW$ on $H$, and a
decreasing filtration $\cF$ on $H_\C$ such that for each $n\in \Z$ the
pair
\[
\left( \Gr_n^{\cW}H, \Gr_n^{\cW_{{\C}}}\cF\right)
\]
is a pure $\R$-Hodge structure of weight $n$.
\end{definition}
\begin{example}\label{kldjldqwdqwdqwd} We define the Tate $\R$-Hodge structure $\R(1)$ as follows. The underlying real vector space is $\R$.
Its weight filtration is given by $$0=\cW_{-3}\subset \cW_{-2}=\R\ .$$ Finally, the Hodge filtration
is given by $$0=\cF^{0}\subset \  \cF^{-1}  \cong \C\ .$$ For every integer $n\in \Z$ we set $\R(n):=\R(1)^{\otimes n}$. 
Note that  the $\R$-Hodge structure $\R(n)$ is pure of weight $-2n$.
In order to be compatible with classical definitions,  
we identify the complexification $\R(1)_{\C}$ with $\C$ by $x\otimes 1\mapsto ix$. 
{For example, this} is  used in order to derive \eqref{wqdqdqwjdhdkjhjkqwdqwdqwdwqdqwd} below.
\end{example}
A morphism of mixed $\R$-Hodge structures $f\colon (H, \cW, \cF) \to (H', \cW', \cF')$ is an $\R$-linear
map $f\colon H\to H'$ compatible with the filtration $\cW$ such that $f \otimes \id_\C\colon H_\C
\to H'_\C$ is compatible with $\cF$.
We denote the category of mixed $\R$-Hodge structures by $\MHS_\R$.
With the natural definition of the tensor product of mixed Hodge structures, $\MHS_\R$ is an
abelian tensor category in which every object has a dual, i.e.~$\MHS_{\R}$ is rigid. Moreover, it is enriched over $\R$-vector spaces. See for example \cite[Ex.~3.2, Cor.~3.9]{Peters-Steenbrink} for details.

By a filtered complex we mean a complex in the category of filtered abelian groups.

\begin{definition}\label{defmhc}
A \emph{mixed $\R$-Hodge complex} is a triple $(C, \cW, \cF)$
where $(C, \cW)$ is an increasingly filtered, bounded complex of $\R$-vector spaces ($\cW$ is
called the \emph{weight filtration}), and $\cF$ (the \emph{Hodge filtration}) turns  $C_\C:=C\otimes_{\R}\C$ into a
decreasingly filtered  complex such that
\begin{enumerate}
\item $H^k(C)$ is a finite dimensional $\R$-vector space for each $k\in \Z$,
\item for every $n\in \Z$ the differential of the filtered complex
\[
\left(\Gr^{\cW_\C}_n C_\C,  \Gr^{\cW_\C}_n\cF\right)
\]
is strict, i.e. the corresponding spectral sequence degenerates at $E_1$,
\item for every $k,n\in \Z$ the $\R$-vector space $H^k(\Gr_n^\cW C)$ equipped with the
filtration on 
\[
H^k(\Gr_n^\cW C) \otimes \C \cong H^k(\Gr_n^{\cW_\C}C_\C)
\]
induced by $\cF$ is a pure $\R$-Hodge structure of weight $n$.
\end{enumerate}

Morphisms of mixed $\R$-Hodge complexes are morphisms of complexes which are compatible with the
filtrations. We denote the category of mixed $\R$-Hodge complexes by $\MHC_\R$.
The tensor product of complexes and filtrations induces a tensor structure on $\MHC_\R$. If $M$ is a mixed Hodge complex, then we denote the underlying complex of real vector spaces by $M_\R$,
the weight filtration by $\cW$, and the Hodge filtration by $\cF$.
\end{definition}
\begin{remark}
{This notion of a mixed Hodge complex differs from Deligne's one \cite[8.1.5]{Deligne-HodgeIII}, who requires  $H^k(\Gr_n^\cW C)$ to be pure of weight $k+n$. But our notion 
 agrees with that of Beilinson \cite[3.2]{BeilinsonHodge}. 
Let $\cW$ be an increasing filtration on a complex $C$. The d\'ecalage of $\cW$ is the  filtration $\widehat{\cW}$ defined by
\begin{equation}\label{eq:decalage}
\widehat{\cW}_{k}C^{n} := \{ x\in \cW_{k-n}C^{k}\,|\, dx\in\cW_{k-n-1}C^{n+1}\}.
\end{equation}
A mixed $\R$-Hodge complex in the sense of Deligne gives one in our sense by replacing the weight filtration by its d\'ecalage.}
\end{remark}

We denote by $\Ch^b(\MHS_\R)$ the category of bounded complexes of mixed $\R$-Hodge structures. We
have a natural inclusion
\begin{equation}\label{laber1}
\Ch^b(\MHS_\R) \hookrightarrow \MHC_\R
\end{equation}
as a full tensor subcategory.  On the domain and target of the map \eqref{laber1} we have notions of quasi-isomorphisms. These are compatible with the map \eqref{laber1} and the respective tensor products. 
On each side, we denote the collection of quasi-isomorphisms by $W$. {Then} \eqref{laber1} induces a map between symmetric monoidal $\infty$-categories 
\begin{equation}\label{ddghjkqwdhkjqwdqwd}
\Ch^b(\MHS_\R)[W^{-1}] \to \MHC_{\R}[W^{-1}]\ 
\end{equation}
(see Appendix \ref{appendix_Kram} for the notation $(-)[W^{-1}]$). The following result is standard, it is for example stated in \cite{Drew}. But for completeness we include the sketch of a proof.
\begin{lemma}\label{lem:stable}
Both $\infty$-categories in \eqref{ddghjkqwdhkjqwdqwd} are stable and the functor is exact.
\end{lemma}
\begin{proof}
Let us give a proof for the second $\infty$-category; the first works similar (but easier), and the exactness of the functor is obvious from the description of (co)limits. We first consider $\MHC_{\R}[H^{-1}]$ where $H$ is the class of chain-homotopy equivalences. Then it follows from \cite[Prop.~1.3.4.7]{HA} (using the same argument as in 
\cite[Prop.~1.3.4.5]{HA}) that $\MHC_{\R}[H^{-1}]$ is equivalent to the $\infty$-category underlying the dg-category $\MHC_{\R}$. But this dg-category is easily seen to be stable, similar to the proof of \cite[Prop.~1.3.2.10]{HA}. 

Now, as a second step we use that  $\MHC_{\R}[W^{-1}] \simeq (\MHC_\R[H^{-1}])[W^{-1}]$. Thus we only need to show that the localization of the stable $\infty$-category $\MHC_\R[H^{-1}]$ at the quasi-isomorphisms remains stable. But this localization 
is the same as the Verdier quotient (discussed in the $\infty$-categorical setting in \cite[Section 5.1]{GBT}) at the full stable subcategory of complexes which are quasi-isomorphic to zero. To see this, note that a morphism is a quasi-isomorphism 
if and only if its cone is quasi-isomorphic to zero and use \cite[Prop.~5.4]{GBT}. Thus, since the Verdier quotient is stable, this finishes the proof.
\end{proof}

The homotopy category of the left-hand side of \eqref{ddghjkqwdhkjqwdqwd} is the bounded derived category {$D^b(\MHS_\R)$} of mixed $\R$-Hodge structures, and it is a result of Beilinson's {\cite[Thm.~3.4]{BeilinsonHodge}} that the induced map
\[
D^b(\MHS_\R)  \iso\Ho(\MHC_\R[W^{-1}])
\]
is an equivalence of categories.
Since both sides are stable according to  Lemma~\ref{lem:stable}, and since equivalences between stable $\infty$-categories can be detected on the the level of homotopy categories, this implies:
\begin{proposition}
The functor 
\begin{equation*} 
\Ch^b(\MHS_\R)[W^{-1}] \to \MHC_{\R}[W^{-1}]\ 
\end{equation*}
from \eqref{ddghjkqwdhkjqwdqwd} is an equivalence of symmetric monoidal stable $\infty$-categories.
\end{proposition}
This has also been considered by Drew~\cite{Drew}.
%
 Since $\MHC_\R[W^ {-1}]$ is a stable $\infty$-category, the homotopy category  $\Ho(\MHC_\R[W^ {-1}])$ is canonically triangulated. In the following we provide an explicit model for the mapping spectra in $\MHC_\R[W^ {-1}]$.

\begin{construction}
The tensor category $\MHC_\R$ is closed. Thus, for mixed Hodge complexes $M, N$, there is an internal hom mixed $\R$-Hodge complex 
$\ihom(M,N)$. It  can be described explicitly as follows.
The underlying complex of real vector spaces $\ihom(M,N)_\R$ is the usual internal hom-complex $\ihom(M_\R, N_\R)$ between complexes of real vector spaces. The weight
filtration is given by 
\[
f\in \cW_k\ihom(M, N)_{\R} \quad \text{ iff } \quad f(\cW_lM_\R) \subseteq \cW_{l+k}N_\R \quad
\text{ for all } l\in\Z
\]
and the Hodge filtration by
\[
f\in \cF^k\ihom(M, N)_{\C} \quad \text{ iff } \quad f(\cF^lM_\C) \subseteq \cF^{l+k}N_\C \quad
\text{ for all } l\in\Z\ .
\]
Here we use the canonical identification
$\ihom(M_\C, N_\C)\cong \ihom(M, N)_{\C} $.

We define a complex $I\Gamma(M,N)$ as follows.
We let $\Omega(I)$ denote the {commutative differential graded algebra (cdga)}  of smooth real valued forms on the unit interval $I:=[0,1]$. For a pair of mixed $\R$-Hodge complexes  $M,N\in \MHC{_{\R}}$ 
we  define\footnote{%
Given a diagram of inclusions of real CDGA's 
$$\xymatrix{&A\ar[d]\\B\ar[r]&C} \ ,$$ the CDGA
$\{\omega\in \Omega(I)\otimes_{\R} {C}\:|\: \omega_{|0}\in A, \omega_{|1}\in B \}$
is a model for the homotopy pull-back in  CDGA's
$A\times^{h}_{C}B$.
Similarly,  
$I\Gamma(M,N)$ is a model for the homotopy pull-back in the derived category of real chain complexes
$$(   {\cW_{0}} \ihom(M,N)_{\R}) \times^{h}_{\cW_{0}  \ihom(M,N)_\C} (  \cW_{0} \ihom(M,N)_{{\C}}\cap \cF^{0}   \ihom(M,N)_{{\C}})$$
which behaves well with respect to composition. }
\begin{equation}\label{eklejlk23je23e32ee32e23e32e32e}
I\Gamma(M,N) := \left\{ \omega\in \Omega(I)\otimes_\R \cW_0\ihom(M,N)_\C \;|\; \omega|_0\in  \ihom(M,N)_\R, \omega|_1\in \cF^0\ihom(M,N)_\C \right\}.
\end{equation}
Given a third mixed  $\R$-Hodge complex $P$, we define a composition $$I\Gamma(N,P) \otimes I\Gamma(M,N) \to I\Gamma(M,P)$$ using the wedge-product on $\Omega(I)$ and the composition of morphisms $$\ihom(N,P)\otimes \ihom(M,N) \to \ihom(M,P)\ .$$
It is not hard to check that this endows $\MHC_{\R}$ with a dg-structure. We denote the resulting dg-category by $\MHC^{I\Gamma}_{\R}$.
The bifunctor $I\Gamma$ is compatible with the formation of tensor products of mixed Hodge complexes.
In this way, $\MHC^{I\Gamma}_{\R}$ inherits the structure of a symmetric monoidal dg-category.
\end{construction}

For any dg-category $\bC$, Lurie constructs in \cite[1.3.1.6]{HA} an $\infty$-category $N_{\dg}(\bC)$ called the dg-nerve of $\bC$. It is equivalent to the $\infty$-category obtained by first applying the Dold-Kan 
correspondence to the Hom-complexes in $\bC$ (truncated at 0) to obtain a simplicially enriched category and then the homotopy coherent nerve. 
One can describe the mapping spaces equivalently as the infinite loop spaces of the Eilenberg-MacLane spectra associated to the full mapping complexes.

\begin{lemma}\label{lemma:IGammaQis}
Let $f\colon M \to N$ be a quasi-isomorphism of mixed Hodge complexes. Then   $f$ is sent to an equivalence in the dg-nerve $N_{\dg}(\MHC_\R^{I\Gamma})$.
\end{lemma}
\begin{proof}
We have to show that $f$ is sent to an isomorphism in the homotopy category $\Ho(N_{\dg}(\MHC_\R^{I\Gamma}))$. By construction of the dg-nerve, its homotopy category is isomorphic to the homotopy category of the dg-category $\MHC_{\R}^{I\Gamma}$ \cite[Rem.~1.3.1.11]{HA}. The morphisms from $M$ to $N$ in this category are given by $H^{0}(I\Gamma(M,N))$. 

We first show that $f$ has a right inverse $g$ in $H^{0}(I\Gamma(N,M))$. Since $f$ is a quasi-isomorphism of mixed Hodge complexes, it is a bifiltered quasi-isomorphism with respect to the weight and Hodge filtrations. This implies that the induced map $f_{*} \colon I\Gamma(N,M) \to I\Gamma(N,N)$ is a quasi-isomorphism of chain complexes. Since both complexes are bounded complexes of real vector spaces, $f_{*}$ is actually a chain homotopy equivalence and thus admits a chain homotopy inverse $\phi$. Set $g := \phi(\id_{N}) \in I\Gamma(N,M)$. This is a $0$-cycle satisfying $[f\circ g] = [f_{*}\circ\phi(\id_{N})] = [\id_{N}] \in H^{0}(I\Gamma(N,N))$, i.e.~$g$ is the desired right inverse.

Dually, one shows that $f$ also has a left inverse in $H^{0}(I\Gamma(N,M))$. Together, this implies that the image of $f$ in $H^{0}(I\Gamma(M,N))$ is an isomorphism.
\end{proof}

We refer to {Appendix}~\ref{appendix_Kram} for an explanation of the assertion that a morphism $\mathbf{C}\to \mathbf{D}$ between $\infty$-categories exhibits $\mathbf{D}$ as a (Dwyer-Kan) localization
$\mathbf{C}[W^{-1}]$.

\begin{proposition}\label{kldnqwkldwqdwqd}
The canonical functor $\MHC_\R \to N_{\dg}(\MHC^{I\Gamma}_\R)$ exhibits $N_{\dg}(\MHC_\R^{I\Gamma})$ as the localization  $\MHC_{\R}[W^{-1}]$.
In particular, since $\MHC_\R \to N_{\dg}(\MHC^{I\Gamma}_\R)$ is symmetric monoidal, we get an equivalence of symmetric monoidal $\infty$-categories 
\[
\MHC_{\R}[W^{-1}] \iso N_{\dg}(\MHC_{\R}^{I\Gamma}).
\]
\end{proposition}
Before proving this, we formulate a corollary. By $H$ we denote the Eilenberg-MacLane functor from chain complexes to spectra, and by $\map(-,-)$ the mapping spectrum between objects of a stable $\infty$-category 
(see Appendix~\ref{qkdhkqwdlwdwqddqwdqd}). 
Proposition~\ref{kldnqwkldwqdwqd} implies in particular that $N_{\dg}(\MHC^{I\Gamma}_{\R})$ is stable and the mapping spectrum $\map(M,N)$ can be computed as $H(I\Gamma(M,N))$.
According to Corollary~\ref{cor:map-calg}, the functor $\map(C,-)$ has a symmetric monoidal refinement for every cocommutative coalgebra object $C$. 
\begin{corollary}\label{kljwdlkqwjdwqdwqdwqdqd} 
For every $C\in \CAlg(\MHC_{\R}[W^{-1}]^{op})$ there is a natural equivalence of lax symmetric monoidal functors 
$$
H (I\Gamma(C,-)) \simeq \map(C,-):
  \MHC_{\R}[W^{-1}] \to \Sp\ .$$
\end{corollary}

\begin{proof}[Proof of Proposition \ref{kldnqwkldwqdwqd}]
By Lemma \ref{lemma:IGammaQis}, the natural map of $\infty$-categories $\MHC_{\R} \to N_{\dg}(\MHC_{\R}^{I\Gamma})$ factors through a map
\[
\MHC_{\R}[W^{-1}] \to N_{\dg}(\MHC_{\R}^{I\Gamma}).
\]
It is essentially surjective. We follow Beilinson's argument in \cite{BeilinsonHodge} in order to show that this map also induces equivalences of mapping spaces. 

We denote the canonical functor $\MHC_{\R} \to \MHC_{\R}[W^{-1}]$ by $\iota$.
Let $M$ and $N$ be mixed Hodge complexes. 
We have to show that the natural map
\[
\Map(\iota(M), \iota(N)) \to \Omega^{\infty}(H(I\Gamma(M,N))
\]
is an equivalence. 
Since $\MHC_{\R}[W^{-1}]$ is stable, we have an isomorphism 
\[
\pi_{i}(\Map(\iota(M),\iota(N))) \cong \pi_{0}(\Map(\iota(M),\iota(N[i]))).
\]
By direct inspection, we also have 
\[
\pi_{i}(\Omega^{\infty}(H(I\Gamma(M,N)))) \cong \pi_{0}(\Omega^{\infty}(H(I\Gamma(M,N[i])))).
\]
Hence it suffices to show that
\begin{equation}\label{eq:pi0MapIGamma}
\pi_{0}(\Map(\iota(M), \iota(N))) \to \pi_{0}(\Omega^{\infty}(H(I\Gamma(M,N)))) \cong H^{0}(I\Gamma(M,N))
\end{equation}
is an isomorphism. The group $\pi_{0}(\Map(\iota(M), \iota(N)))$ is the group of morphisms from $\iota(M)$ to $\iota(N)$ in the homotopy category $\Ho(\MHC_{\R}[W^{-1}])$. This homotopy category is equivalent to the 1-categorical localization of the 1-category of mixed $\R$-Hodge complexes $\MHC_{\R}$  by the class of quasi-isomorphisms, which we denote by $D(\MHC_{\R})$. We now use the fact that morphisms in  $D(\MHC_{\R})$ can be computed by calculus of fractions. More precisely, we let $\MHK_{\R}$ denote the 1-category of mixed Hodge complexes where morphisms are  chain homotopy classes of morphisms between mixed Hodge complexes. 
This is a triangulated category, and $H^{*}: \MHK_{\R} \to \MHS_{\R}$ is a cohomological functor. Therefore, the class of quasi-isomorphisms admits a calculus of fractions \cite[10.4.1]{WeibelHomo},
and $D(\MHC_{\R})$ is the corresponding Verdier localization of $\MHK_{\R}$.

In order to describe morphisms in $D(\MHC_{\R})$ explicitly, we consider the map
\begin{equation}\label{eq:xi}
\xi_{M,N}\colon \cW_0\ihom(M,N)_\R\oplus \cW_0\cap\cF^0\ihom(M,N)_\C \to \cW_0\ihom(M,N)_\C
\end{equation}
in $\Ch(\Ab)$
which is the difference of the two obvious inclusions, and we set 
\begin{equation}\label{eq:def-Hom}
\underline{\Hom}(M,N) := \ker(\xi_{M,N}) \in \Ch(\Ab).
\end{equation}
Essentially by definition we then have
\[
\Hom_{\MHK_{\R}}(M,N) \cong H^{0}\left(\underline{\Hom}(M,N)\right).
\]
Let $\bI_N$ the subcategory of $\MHK_{\R}\backslash N$ consisting of quasi-isomorphisms $N \xrightarrow{\sim} N'$.
Calculus of fractions  yields
\begin{equation}\label{eq:calc-of-fractions}
\pi_{0}(\Map(\iota(M), \iota(N))) \cong 
\Hom_{D(\MHC_{\R})}(M,N)\cong \colim_{\bI_{N}} \Hom_{\MHK_{\R}}(M,N'). 
\end{equation}
We now compute the colimit on the right-hand side.
We set 
\[
\Gamma(M,N) := \Cone(\xi_{M,N})[-1] \in \Ch^b(\Ab).
\]
Since quasi-isomorphisms of mixed Hodge complexes are  automatically bifiltered
quasi-isomorphisms, the functor $\Gamma(M, -)$ sends quasi-isomorphisms in $\MHC_\R$ to quasi-iso\-mor\-phisms
in $\Ch^b(\Ab)$.
\begin{lemma}\label{sep2602}
 {There exists} an exact triangle  in the derived category of abelian groups:
\[
  \ker(\xi_{M,N}) \to \Gamma(M,N) \to  \coker(\xi_{M,N})[-1] {\to} \ker(\xi_{M,N})[1] .
\]
\end{lemma}
\begin{proof}
More generally, if $\xi\colon A \to B$ is a map of complexes of abelian groups, we have  
an
exact triangle
\[
\ker(\xi) \xrightarrow{{\mathrm{inc}}} \Cone(A \xrightarrow{\xi} B)[-1] \to \Cone({\mathrm{inc}}) \to \ker(\xi)[1].
\]
On the other hand, from the short
exact sequence
\[
0 \to \ker(\xi) \xrightarrow{{\mathrm{inc}}} \Cone(A \xrightarrow{\xi} B)[-1] \to \Cone(A/\ker(\xi)
\xrightarrow{\xi} B)[-1]
\to 0
\]
we get quasi-isomorphisms 
\[
\Cone({\mathrm{inc}}) \xrightarrow{\simeq} \Cone(A/\ker(\xi)
\xrightarrow{\xi} B)[-1]
\]
and, since $A/\ker(\xi) \xrightarrow{\xi} B$ is injective,
\[
\Cone(A/\ker(\xi) \xrightarrow{\xi} B)[-1] \xrightarrow{\simeq} \coker(\xi)[-1]\ . \qedhere 
 \]
\end{proof}
   
\begin{lemma}\label{sep2603}
For any $i$ we have
\[
\colim_{\bI_N} H^{i}(\coker(\xi_{M,N'})) = 0.
\]
\end{lemma}
\begin{proof}

We consider a class $[\alpha]\in H^{i}(\coker(\xi_{M,N^{\prime}}))$. First we choose a representative
 $$u=(u_j)_{j\in\Z} \in \cW_0\ihom^i(M,N')_\C=\prod_{j\in\Z} \cW_0\Hom(M^j,
N^{\prime, i+j})\ .$$ 
We now define a mixed $\R$-Hodge complex $N^{\prime\prime}$ as follows. The underlying complex of real vector spaces
is 
\[
N^{\prime\prime}_\R:=\Cone(M_\R[-i-1]\xrightarrow{0\oplus\id} N'_\R\oplus M_\R[-i-1]),
\]
the weight filtration is {defined component-wise}, and the Hodge filtration is given as follows. We
define
\[
\cF^k N^{\prime\prime,i+j} \subseteq N^{\prime\prime,i+j}= M^j \oplus N^{\prime,i+j} \oplus M^{j-1} 
\]
to be the subspace generated by
\begin{align*}
&(0,x,0), && x\in \cF^kN^{\prime,i+j}\\
&(y,u_j(y),0), && y\in\cF^kM^j\\
&(dz, -du_{j-1}(z), z), && z\in \cF^kM^{j-1}.
\end{align*}
{We now} show that $N^{\prime\prime}$ is a mixed $\R$-Hodge complex, and that
the natural inclusion $\phi:N'\to N^{\prime\prime}$ is an equivalence. Both claims will follow from the fact that  $\phi$ is a bifiltered quasi-isomorphism. This, in turn, follows from the observation
 that for every pair of integers $k,l$ we have an induced exact sequence
\[
0 \to \cW_l\cap\cF^kN' \to \cW_l\cap\cF^kN'' \to \Cone\left(\cW_l\cap\cF^kM \xrightarrow{\id}
\cW_l\cap\cF^kM\right)[-i-1] \to 0.
\]
By definition of the Hodge filtration on $N''$, the map $(\id,u_j,0)\colon M^j \to N^{\prime\prime,i+j}$
respects the Hodge filtration. Hence
\[
v:=-\left((\id,0,0)\oplus (\id,u_j,0)\right)_j \in \cW_0\ihom^i(M_\R,N''_\R) \oplus
\cW_0\cap\cF^0\ihom^i(M,N'')
\]
satisfies $\xi_{M,N''}(v) = \phi\circ u$.  
Therefore $\phi([\alpha])=0$. 
\end{proof}

Putting together {\eqref{eq:calc-of-fractions}} and Lemmas \ref{sep2602} and \ref{sep2603} we get natural
isomorphisms
\begin{equation}\label{oct2301}
\pi_{{0}}(\Map(\iota(M),\iota(N))) \cong \colim_{\bI_N} H^{{0}}(\Gamma(M,N')) \cong H^{{0}}(\Gamma(M,N)).
\end{equation}
Hence, to finish the proof of Proposition \ref{kldnqwkldwqdwqd}, it suffices to establish the following lemma.
\begin{lemma}\label{dhwqdkjqwhdkqwdwqdd}
Given mixed Hodge complexes $M,N$, there is a natural quasi-isomorphism
\[
I\Gamma(M,N) \xrightarrow{\sim} \Gamma(M,N)
\]
which is compatible with the obvious inclusion of $\underline{\Hom}(M,N)$ (see~\eqref{eq:def-Hom}) on both sides.
\end{lemma}
\begin{proof}
We define $q\colon I\Gamma(M,N) \to \Gamma(M,N)$ by the formula
\[
q(\omega) := \left(\omega|_{0}\oplus \omega|_{1}, -\int_{0}^{1}\omega\right).
\]
One sees using \eqref{eklejlk23je23e32ee32e23e32e32e} that this indeed defines an element in $\Gamma(M,N)=\Cone(\xi_{M,N})$ (see \eqref{eq:xi}).
The proof that this is a quasi-isomorphism is standard (see \cite[2.6]{buta2}).
\end{proof}
This concludes the proof of Proposition~\ref{kldnqwkldwqdwqd}.
\end{proof}

\begin{example}\label{ijqwdiwqjdwqdljqwldkwqdqwdwqdwqd}
In this example we calculate for all $p,q\in \Z$  the homotopy groups of the spectrum
${\map}(\iota(\R(p)),\iota(\R(q)))$.
We first of all observe that
$$
{\map}(\iota(\R(p)),\iota(\R(q)))\simeq {\map}(1,\iota(\R(q-p)))\ .
$$
It therefore suffices to calculate
$ {\map}(1,\iota(\R(p)))$ for all $p\in \Z$.
We now use that by Corollary \ref{kljwdlkqwjdwqdwqdwqdqd} this mapping spectrum is given by 
$${\map}(1,\iota(\R(p)))\simeq H(I\Gamma(1,\R(p)))\ .$$ Using the  explicit formula 
 \eqref{eklejlk23je23e32ee32e23e32e32e} for $I\Gamma$ {and the definition of $\R(p)$ from Example \ref{kldjldqwdqwdqwd}} we get
$$
I\Gamma(1,\R(p))=\{\omega\in \Omega(I)\otimes_{\R}{\cW}_{2p}\C\:|\: \omega|_{0}\in i^{p}\R\ ,   \omega|_{1}\in \cF^{p}\}\ .
$$
Here the weight filtration of $\C$ is given by ${\cW}_{-1}\C=0$ and ${\cW}_{0}\C=\C$, and the Hodge filtration is $\cF^{0}\C=\C$ and $\cF^{1}\C=0$.
We conclude:
\[
\pi_{k}({\map}(1,\iota(\R(p))))= 
\begin{cases}
0 & \text{if } p<0,\\
0 & \text{if } p=0, k\not= 0, \\
\R & \text{if } p=0, k=0,\\
0 & \text{if } p>0, k\not= -1,\\
i^{p-1}\R & \text{if } p>0, k=-1.
\end{cases}
\]
Observe that by this calculation the mapping spaces 
${\Map}(\iota(\R(p)),\iota(\R(q)))$ are discrete.
See \cite[Ex.~3.34]{Peters-Steenbrink} for this and further examples.
\end{example}

\section{The complex \texorpdfstring{$\IDR$}{\bf IDR} and absolute Hodge cohomology}

\newcommand{\cE}{\mathcal{E}}

The goal of this  section is to construct a functor $\IDR$ from smooth algebraic varieties over $\C$ to differential graded algebras whose cohomology groups compute absolute Hodge cohomology.

Let $\Ind(\MHC_{\R})$ denote the Ind-completion of $\MHC_{\R}$ as discussed in the appendix after Definition~\ref{def:Ind-completion}.  
We consider the symmetric algebra 
$$
T:=\Symm(\R(1)[2])\in \CAlg(\Ind(\MHC_{\R}))
$$ 
on the Tate $\R$-Hodge structure of weight {$-2$} (see Example \ref{kldjldqwdqwdqwd})  considered as a mixed $\R$-Hodge complex concentrated in cohomological degree $-2$.  
The underlying  mixed $\R$-Hodge complex of $T$ is \begin{equation}\label{jkhjkshkjqwsqwswqsqs}
T\cong {\bigoplus}_{p\ge 0} \R(p)[2p]\in \Ind(\MHC_{\R})\ .
\end{equation} 
The appearance of this infinite coproduct forces us to work in the Ind-completion. We consider objects of $\MHC_{\R}$ like $\R(p)[2p]$ as objects of the Ind-completion without further notice. 
We furthermore set 
\begin{equation}\label{ddehjwedgewjhdewdwed}
\bbT:=\iota(T)\in  \CAlg(\Ind(\MHC_{\R}[W^{-1}]))\ .
\end{equation}
To make sense of the last definition, we have implicitly used that there is a canonical morphism $\Ind(\MHC_{\R}) \to \Ind(\MHC_{\R}[W^{-1}])$ which is symmetric monoidal. 
See Appendix~\ref{lkjdqlkwdjlwqdwqdqdq} for details.
For every commutative algebra $D\in \CAlg(\MHC_{\R})$ we can define the commutative ring spectrum
$$
\map(1,\bbT\otimes \iota(D) )\in \CAlg(\Sp)
$$
{(see Corollary~\ref{cor:map-calg}).}
On the other hand, {the lax symmetric monoidal functor $$I\Gamma(\R(0), -) \colon \MHC_{\R} \to \Ch$$ extends to a lax symmetric monoidal functor $\Ind(\MHC_{\R}) \to \Ch$ and hence}
we can define the cdga 
$$
\cE(T \otimes D) := I\Gamma(\R(0),T\otimes D) \in \CAlg(\Ch)\ .
$$
Since $H$ commutes with filtered colimits, Corollary~\ref{kljwdlkqwjdwqdwqdwqdqd} implies that
we have a natural equivalence
\begin{equation}\label{geregregeggr34324324}
 H(\cE(T \otimes D))\simeq \map(1,\bbT\otimes \iota(D) )\ .
\end{equation}
Using \eqref{jkhjkshkjqwsqwswqsqs} we have an isomorphism of chain complexes
\begin{equation}\label{ndqwdnqwdqwdwq}
\cE(T \otimes D)\cong {\bigoplus}_{p\ge 0}  \cE(D(p)[2p])
\end{equation}
where here and in the following we use the abbreviation 
\begin{equation}\label{eq:tate-twist} 
D(p)[2p] := D \otimes \R(p)[2p].
\end{equation}

We now recall the definition of absolute Hodge cohomology, which was introduced by Beilinson \cite{BeilinsonHodge}.
We let $\Sm_{\C}$ be the category of smooth algebraic varieties over {$\C$}.
To any $X$ in $\Sm_\C$ one can associate a mixed Hodge complex 
\[
A_{\log}(X)=(A_{\log}(X)_{\R},\cW,\cF)
\]
({see} \cite[(1.1)]{Burgos-Wang} where this complex is denoted by $E^*_{\log}(X)$ and the weight filtration by $\widehat{W}$. {This} is the d\'ecalage of the weight filtration introduced by Deligne.). 
 Its underlying complex $A_{\log}(X)_{\C}$ is a subcomplex of {the complex of smooth $\C$-valued differential forms on the complex manifold $X(\C)$. A} form $\omega$ belongs to this subcomplex if  it extends to {a smooth} form  with logarithmic singularities on some smooth compactification $\overline{X}$ of $X$ such that $\overline{X}-X$ is a divisor with normal crossings. {Note that the choice of compactification may depend on $\omega$.} The cohomology groups of $A_{\log}(X)$ are the Betti cohomology groups $H^*(X(\C),\R)$ together with the mixed Hodge structure introduced by Deligne \cite{Deligne-HodgeII}.
The wedge product of forms turns $A_{\log}(X)$ into an object of $\CAlg(\MHC_{\R})$.
The definition of  $A_{\log}(X)$ can be obtained by specializing the more general definition of $A_{\log}(M\times X)$ given in Remark \ref{ldjqldqwkdhkd} below to the case where $M=*$.

The absolute Hodge cohomology spectrum of $X$ is defined in terms of a mapping spectrum of the stable $\infty$-category 
$\Ind(\MHC_{\R}[W^{-1}])$ as 
$$\map(1,\bbT\otimes \iota(A_{\log}(X)))\in \CAlg(\Sp)\ .$$
By   \eqref{geregregeggr34324324} we have
an equivalence of commutative ring spectra
$$
H(\cE({T\otimes}A_{\log}(X)))\simeq\map(1,\bbT\otimes \iota(A_{\log}(X)))\ .$$
The decomposition \eqref{ndqwdnqwdqwdwq} gives a decomposition of chain complexes
$$\cE(T\otimes A_{\log}(X))\cong \bigoplus_{p\ge 0} \cE(A_{\log}(X)(p)[2p])\ .$$ 
The absolute Hodge cohomology groups are now given by
\[
H^{2p+*}_{\absHodge}(X,\R(p)):=H^{*}(\cE(A_{\log}(X)(p)[2p]))
\]
so that we have
\begin{equation}\label{eq:def-abs-Hodge}
\pi_{*}(\map(1,\bbT\otimes \iota(A_{\log}(X))))=\bigoplus_{p\in \nat} H^{2p-*}_{\absHodge}(X,\R(p))\ .
\end{equation}

\begin{remark}
It is our philosophy that algebraic $K$-theory classes  on a smooth variety $X$ are realized in terms of vector bundles parametrized by auxiliary smooth manifolds, and that
regulator classes are represented by characteristic forms associated to geometric structures on these bundles, see e.g.  \cite{buta}.  In general, the complex $\cE(T\otimes A_{\log}(X))$ is 
too small to contain these characteristic forms. The natural home of these characteristic forms is the differential graded algebra $\IDR(X)$ which we now introduce. 
\end{remark}
\begin{remark}\label{ldjqldqwkdhkd} {We first recall the extension of $A_{\log}$ to a presheaf on $\Mf\times\Sm_{\C}$ where $\Mf$ denotes the category of smooth manifolds. Let $M$ be a manifold and $X$ a smooth variety over $\C$. For a fixed smooth compactification $X \hookrightarrow \overline{X}$ such that $D:=\overline{X}-X$ is a divisor with normal crossings, called good compactification in the following, we define 
\[
A_{M\times\overline{X},\R}(M\times X, \log D) \subseteq A(M\times X)_{\R}
\]
to be the subcomplex locally generated as an algebra over $A(M\times \overline{X})_{\R}$ by 1 and
\begin{equation}\label{eq:wt1}
\log(z_{i}\bar z_{i}),\ \Re\frac{dz_{i}}{z_{i}},\ \Im \frac{dz_{i}}{z_{i}}
\end{equation}
where the $z_{i}, i\in I,$ are local coordinates on $\overline{X}$ defining $D$ locally by the equation $\prod_{i\in I}z_{i}=0$.
The naive weight filtration $\cW'$ is the multiplicative increasing filtration by  $A(M\times\overline X)_{\R}$-modules obtained by assigning weight $0$ to the section $1$  and weight $1$ to the sections listed in  \eqref{eq:wt1}.
We define a decreasing filtration $\mathcal{L}$ of $A_{ {M\times \overline X,\R}}(M\times X,  {\log D})$ such that $\mathcal{L}^pA_{ {M\times \overline X,\R}}(M\times X,  {\log D})$ is the subcomplex of differential forms that are given locally by
\[
\sum _{I,J,K,|I|\ge p} \omega_{I,J,K} \:dx^{I}\wedge \Re dz^{J}\wedge \Im dz^{K}\ ,
\]
where the $x_i$ are local coordinates for $M$, the $z_j$ local coordinates for $\overline X$, and $\omega_{I,J,K}$ local  smooth functions on  $M\times X$.  
We define $\cW''$ as the diagonal filtration of $\cW'$ and $\mathcal{L}$:
\begin{equation*}
\cW''_k A_{M\times\overline X,\R}(M\times X,\log D) := \sum_p  \cW'_{k+p} \cap \mathcal{L}^p A_{M\times\overline X,\R}(M\times X,\log D).
\end{equation*}
Finally, we define the weight filtration $\cW:= \widehat{\cW''}$ as the d\'ecalage of the filtration $\cW''$ (see~\eqref{eq:decalage}).
}
\end{remark}
We further define the complex dg-algebra $$A_{M\times \overline X}(M\times X,\log D) := A_{M\times\overline X,\R}(M\times X,\log D) \otimes_{\R} \C$$ with the induced weight filtration. This complex carries the  decreasing Hodge filtration $\cF$ such that the elements of  $\cF^pA_{M\times\overline X}(M\times X,\log D)$   are locally of the form
\[
\sum _{I,J,K,|J|\ge p} \omega_{I,J,K} \:dx^{I}\wedge dz^{J}\wedge d\bar z^{K},
\]
where the $x_i$ and $z_j$ are local coordinates of $M$ and $X$, respectively.

{Eventually, we define the cdga
\[
A_{\log}(M\times X) := \colim_{X\hookrightarrow \overline{X}} A_{M\times\overline{X}}(M\times X,\log(\overline{X}-X))
\]
with its induced real subalgebra, weight and Hodge filtration, where the colimit runs over all good compactifications of $X$.
}

\begin{definition}\label{dqqlkwdqwdwqdwqd}
For $X\in \Sm_\C$ and an integer $p$ we define the complex $\IDR(p)(X)$ as follows:
\[
\IDR(p)(X) := \left\{\omega\in \cW_{2p} A_{\log}(I\times X)[2p]\:|\: \omega|_0\in (2\pi i)^pA_{\log}(X)_{\R}\ , \omega|_1\in \cF^pA_{\log}(X)_{\C}\right\}
\]
We define $$\IDR(X) := \prod_{p\geq 0}\IDR(p)(X)\ .$$ The wedge product induces maps
$$\IDR(p)(X)\otimes \IDR(q)(X)\to \IDR(p+q)(X)$$ which 
turn $\IDR(X)$ into a differential graded algebra.
\end{definition}
 For negative $p$, the complex $\IDR(p)(X)$ vanishes (since there are no forms of negative weight).
 In the following statement the term `natural' means natural in $X$.
\begin{proposition}\label{ednqwdkjweqdqwdqwd}
We have a natural quasi-isomorpism of commutative differential graded algebras
$$\cE(T\otimes A_{\log}(X))\stackrel{\sim}{\to} \IDR(X)\ .$$
Consequently, we have a natural equivalence of commutative ring spectra
$$
\map(1,\bbT\otimes \iota(A_{\log}(X)))\simeq H(\IDR(X))\ .
$$
\end{proposition}
\begin{proof} 
 Note that using the Tate twist \eqref{eq:tate-twist}  
we can write
\begin{equation}\label{wqdqdqwjdhdkjhjkqwdqwdqwdwqdqwd}
\IDR(p)(X)\cong  \left\{\omega\in  \cW_{0}A_{\log}(I\times X)(p)[2p]\:|\: \omega|_0\in A_{\log}(X)(p)_{\R}\ , \omega|_1\in \cF^0A_{\log}(X)(p)_{\C}\right\}.
\end{equation}
In view of \eqref{eklejlk23je23e32ee32e23e32e32e} the difference between
$\cE(A_{\log}(X)(p)[2p])$ and $\IDR(p)(X)$ is that in the former we use the algebraic tensor product
$\Omega(I)\otimes A_{\log}(X)$, while the latter is based on the larger space of forms $A_{\log}(I\times X)$ on $I\times X$.
The natural inclusion
$$
\cE(A_{\log}(X)(p)[2p])\to \IDR{(p)(X)}
$$
 is a quasi-isomorphism
since both complexes are naturally quasi-isomorphic to
$$
\Cone\left(\left(\cW_{0}A_{\log}(X)(p)_{\R}\oplus \cW_{0}\cap\cF^{0}A_{\log}(X)(p)\right)\to \cW_{0}A_{\log}(X) \right)[2p-1]
$$
(the proof of Lemma \ref{dhwqdkjqwhdkqwdwqdd} applies to $\cE(A_{\log}(X)(p)[2p])$   and $\IDR(p)(X)$ as well).
If we sum {up} these equivalence over all $p\ge 0$, then we get an
induced quasi-isomorphism
$$\cE(T\otimes A_{\log}(X))\to \IDR(X)$$
of commutative differential graded algebras.
\end{proof}

\section{The category of motivic spectra}

In this section, we recall the definition of the symmetric monoidal $\infty$-category of motivic spectra $\MotSp$ following Robalo \cite{Robalo}. We shall only need the universal property of this category which we state below 
as Proposition \ref{univ}. The reader who does not want to get involved into the exact construction of this category can take this statement as a definition  of the category of motivic spectra.  

Let $Y\colon \Sm_\C\to \Fun(\Sm_\C^{op}, \sSet[W^{-1}])$ be the  Yoneda embedding (composed with the identification of sets with constant simplicial sets which is not noted explicitly).
It is symmetric monoidal with respect to the cartesian symmetric monoidal structures. We denote by $$\Fun^{Nis}(\Sm_{\C}^{op}, \sSet[W^{-1}]) \subset \Fun(\Sm_\C^{op}, \sSet[W^{-1}])$$ the full subcategory of sheaves for the Nisnevich topology. This inclusion {fits into an adjunction
\[
L^{Nis}\colon \Fun(\Sm_\C^{op}, \sSet[W^{-1}]) \leftrightarrows  \Fun^{Nis}(\Sm_\C^{op}, \sSet[W^{-1}])\colon inclusion\ ,
\]
where 
the sheafification functor $L^{Nis}$} is symmetric monoidal with respect to the cartesian structures.
We denote by $$\MotSpc := \Fun^{Nis,\bbA^1}(\Sm_\C^{op}, \sSet[W^{-1}])$$ the full subcategory of $\bbA^1$-invariant 
objects, i.e.~objects $E$ for which the natural morphism   $E(\bbA^{1}\times X) \to E(X)$ is an equivalence for all $X\in \Sm_{\C}$, and call it the $\infty$-category of motivic spaces.
Again, this  inclusion  fits into an adjunction
$$L^{\bbA^1}\colon\Fun^{Nis}(\Sm_\C^{op}, \sSet[W^{-1}])\leftrightarrows \MotSpc \colon inclusion\ ,
$$
where again the functor
 $L^{\bbA^1}$ is symmetric monoidal. 
 We denote by $\MotSpc_{\ast}$ the $\infty$-category of pointed objects in $\MotSpc$, and by
$$(-)_+\colon \MotSpc \to \MotSpc_{\ast}$$ the functor that adds a disjoint base point. The cartesian symmetric monoidal
structure on $\MotSpc$ induces a symmetric monoidal structure $\wedge$, called smash product, on $\MotSpc_{\ast}$ such that
$(-)_+$ is monoidal. 

\begin{definition} ({See} \cite[5.10]{Robalo})
The symmetric  {monoidal} $\infty$-category of  motivic spectra $\MotSp$ is obtained from $\MotSpc_{\ast}$ by inverting the
pointed projective line $(\P^1_\C,\infty)$ with respect to the smash product as explained in Appendix \ref{localization}. 
\end{definition}

Note that since $(\P^1_\C,\infty)$ is cyclically invariant (see Definition~\ref{def:cyclic-invariant} and Example~\ref{ex:symmetric-monoidal-localization} for this notion and \cite[Lemma 3.13]{Jardine} for a proof of this assertion), this in particular means that the underlying $\infty$-category of $\MotSp$ is the colimit in presentable $\infty$-categories
$$
\MotSp := \colim( \MotSpc_{\ast} \xrightarrow{\wedge (\P^1_\C,\infty)}\MotSpc_{\ast} \xrightarrow{\wedge (\P^1_\C,\infty)} \MotSpc_{\ast}  \xrightarrow{\wedge (\P^1_\C,\infty)} \ldots \ ) \ .
$$
By construction, $\MotSp$ is a stable presentable symmetric monoidal $\infty$-category. According to Robalo \cite[before Cor.~5.11]{Robalo} its associated homotopy category is the stable motivic 
homotopy category $\SH_\C$ constructed by Morel and Voevodsky. {The symmetric monoidal $\infty$-}category {$\MotSp$} is characterized by a  universal property.  In order to state this property in Proposition \ref{univ}, let 
$\Fun^{L,\otimes}$ denote the category of symmetric monoidal functors that preserve colimits (equivalently, that are left adjoint{s}).
Let $\Sm_{\C,\ast}$ denote the category of pointed smooth schemes over $\Spec(\C)$.
\begin{proposition}[{{\cite[Cor.~5.11]{Robalo}}}]
\label{univ}
The category of motivic spectra $\MotSp$ is a presentably symmetric monoidal $\infty$-category (see \ref{pressymmon} on page \pageref{pressymmon} for this notion) together with a symmetric monoidal functor $\Sigma^\infty : \Sm_{\C,\ast} \to \MotSp$ which satisifies the following universal property:
for every other presentably
symmetric monoidal $\infty$-category ${\bC}$ which is pointed, the induced functor
$$
\Fun^{L,\otimes}(\MotSp, {\bC}) \to \Fun^{\ast,\otimes}(\Sm_{\C,\ast}, {\bC})
$$
is fully faithful and the essential image consists of functors $F: \Sm_{\C,\ast} \to {\bC}$ that are $\bbA^{1}$-invariant, satisfy Nisnevich codescent,
and send $(\P^1,\infty)$ to a tensor invertible object in ${\bC}$. 
\end{proposition}

\section{The motivic \texorpdfstring{$K$}{\bf K}-theory spectrum}\label{lkqlwdwqdqwdwqdwqdwqdw}

We now introduce the motivic spectrum representing algebraic $K$-theory.
We will define {the motivic algebraic $K$-theory spectrum} using the motivic Snaith theorem due to Gepner-Snaith
\cite{Gepner-Snaith} and Spitzweck-{\O}stv{\ae}r \cite{Spitzweck-Ostvaer}.

For each integer $n$ we consider the scheme $\P^n:= \P^n_\C$ in $\Sm_\C$. Abusing notation, we
denote its image {$L^{\bbA^{1}}\circ L^{Nis}\circ Y(\P^{n})$} in $\MotSpc$ by the same symbol. We write $\P^{\infty}:=\colim_n \P^n$.
We denote by $\G_m:=\G_{m,\C}$ the multiplicative group scheme in $\Sm_\C$. {Since $\G_{m}$ is commutative, we can} view the classifying
space $B\G_m$ as an object in $\Fun(\Sm_\C^{op},\CommGrp(\sSet[W^{-1}]))$. Again, we denote its
image in $\MotSpc$ by the same symbol. The following result is well-known.

\begin{lemma}\label{lem:P1BGm}
In $\MotSpc$
we have a natural equivalence
\[
B\G_m \simeq \P^{\infty}.
\]
In particular, $\P^{\infty}$ {can} naturally {be refined to} an object in $\CommGrp({\MotSpc})$.
\end{lemma}
\begin{proof}
A local section cover of $X\in \Sm_{\C}$ is by definition a surjective family of smooth morphisms $\{X_i \to X\}_{i\in I}$ which can be refined by some Nisnevich covering of $X$. Local section covers form a pretopology which generates the Nisnevich topology.
Given a local section cover of $X$, the natural map from the \v{C}ech nerve of the cover to $X$ is an equivalence in $\MotSpc$.

{In $\Sm_{\C}$ we have isomorphisms}
\begin{align}
\begin{split}\label{oct2402}
\P^{n-1} &\cong \colim\left(\cdots (\A^n\setminus
0)\times_{\P^{n-1}} (\A^n\setminus 0) \rightrightarrows (\A^n\setminus 0)\right) \\
& \cong \colim\left(\cdots (\A^n\setminus
0)\times \G_m \rightrightarrows (\A^n\setminus 0)\right)\\
& \cong (\A^n\setminus 0)/\G_m.
\end{split}
\end{align}
{Since the Nisnevich topology is subcanonical, and since the inclusion of sets in simplicial sets as constant simplicial sets is a right adjoint, the Yoneda embedding gives an embedding
\[
\Sm_{\C} \to \Fun^{Nis}(\Sm_{\C}^{op},\sSet[W^{-1}]).
\]
Since for any integer $n$, the projection $\A^n\setminus 0 \to \P^{n-1}$ is a local section cover, the colimit in \eqref{oct2402} is preserved under this embedding.
}
Since $L^{\A^1}$ preserves colimits, we
deduce that \eqref{oct2402} {also} holds in $\MotSpc$. Taking the colimit as $n$ goes to $\infty$, we
deduce that
\[
\P^{\infty} \simeq \colim\left(\cdots (\A^{\infty}\setminus
0)\times \G_m \rightrightarrows (\A^{\infty}\setminus 0)\right) \simeq (\A^{\infty}\setminus 0)/\G_m
\]
in $\MotSpc$. Now the usual arguments show that $\A^{\infty}\setminus 0$ is $\A^1$-contractible. It
follows that $(\A^{\infty}\setminus 0)/\G_m \simeq \ast/\G_m \simeq B\G_m$ as desired.
\end{proof}

We denote  the composition of functors $$\MotSpc \xrightarrow{(-)_+} \MotSpc_{\ast}
\xrightarrow{\Sigma^{\infty}} \MotSp$$ by $\Sigma^{\infty}_+$.
Using the refinement of $\P^{\infty}$ to a commutative monoid  we can view $\Sigma^{\infty}_+\P^{\infty}$ as an object in $\CAlg(\MotSp)$.
We have two morphisms 
\[
\xi\ , 1 \colon \Sigma^{\infty}_+ \P^1 \to \Sigma^{\infty}_+\P^{\infty}.
\]
in $\MotSp$.  The one denoted by $\xi$ is induced by  
the inclusion $\P^1\to \P^{\infty}$.  The other {one}, denoted by $1$, is induced by 
the constant map $\P^1 \to \P^{\infty}$ with value $\infty$. Since
$\MotSp$ is stable, we can consider the difference morphism  
\begin{equation} \label{eq:1-xi}
1-\xi \colon \Sigma^{\infty}_+\P^1 \to
\Sigma^{\infty}_+\P^{\infty}
\end{equation}
 in $\MotSp$. The pointed object $(\P^1,\infty)\in \MotSpc_{\ast}$ is the push-out
\[
\xymatrix{
\infty_+ \ar[r]\ar[d] & \P^1_+\ar[d] \\
\ast \ar[r] & (\P^1,\infty).
}
\]
In other words, $\Sigma^\infty(\P^1,\infty)$ is the cofibre of the map $\Sigma^\infty_+\infty \to
\Sigma^\infty_+\P^1$. It follows that $1-\xi$ factors {essentially} uniquely through
a map 
\[
\beta\colon \Sigma^{\infty}(\P^1,\infty) \to \Sigma^{\infty}_+\P^{\infty}.
\]
We then consider the localization $\Sigma^{\infty}_+\P^{\infty}[\beta^{-1}]$ in the category of motivic spectra as explained in Appendix \ref{localization}. In particular, 
this is an object of $\CAlg(\Mod(\Sigma^{\infty}_+\P^{\infty}))$
whose underlying motivic spectrum is given by the colimit
\[
\Sigma^{\infty}_+\P^{\infty}[\beta^{-1}] \cong \colim\left(\Sigma^{\infty}_+\P^{\infty}
\xrightarrow{\mu_\beta} \Sigma^{\infty}(\P^1,\infty)^{-1} \wedge \Sigma^{\infty}_+\P^{\infty} 
\xrightarrow{\id\wedge\mu_\beta} \dots \right)\ .
\] 
Here \begin{equation}\label{dqwdqwdkjqkjkhkqwdqwd}
\mu_{\beta}:\Sigma^{\infty}_+\P^{\infty}\to  \Sigma^{\infty}(\P^1,\infty)^{-1} \wedge\Sigma^{\infty}_+\P^{\infty}\end{equation}
is the adjoint of the composition \begin{equation}\label{dwedklewdnlkedewddewdwed} 
 \Sigma^{\infty}(\P^1,\infty)\wedge \Sigma^{\infty}_+\P^{\infty}\xrightarrow{\beta\wedge\id: }  \Sigma^{\infty}_+\P^{\infty}\wedge  \Sigma^{\infty}_+\P^{\infty}\xrightarrow{mult} \Sigma^{\infty}_+\P^{\infty}\ .
\end{equation}
We now use the lax symmetric monoidal {forgetful} functor \begin{equation}\label{jkjelkqwjelkqwewqe}\cF:\Mod(\Sigma^{\infty}_+\P^{\infty})\to \MotSp\end{equation} in order to define the motivic $K$-theory spectrum
 \[
\bK := \cF(\Sigma^{\infty}_+\P^{\infty}[\beta^{-1}]) \in \CAlg(\MotSp).
\]
Note that the morphism \eqref{dwedklewdnlkedewddewdwed} induces a morphism
\begin{equation}\label{dewdewkjdnewkdedewdd}
\nu_\beta:\Sigma^{\infty}(\P^{1},\infty)\wedge \bK\to \bK
\end{equation}
in $\MotSp$. 

We have the motivic Snaith theorem (\cite[Thm.~4.17]{Gepner-Snaith} or
\cite[Thm.~1.1]{Spitzweck-Ostvaer} together with \cite[Thm.~4.3.13]{Morel-Voevodsky}):
\begin{thm}\label{dlkqwkldqdwqdwqdwqd}
For any $X\in \Sm_\C$ and any $i \in \Z$ we have an
isomorphism
\[
\pi_i\left(\map(\Sigma^{\infty}_+X, \bK)\right) \cong K_i(X)
\]
where $K_i(X)$ denotes higher algebraic $K$-theory as defined by Quillen or
Thomason-Trobaugh.
\end{thm}

We denote the category of spectra by $\Sp$.
The result{s} of Appendix \ref{qkdhkqwdlwdwqddqwdqd} and the fact that $\bK$ is a commutative algebra object imply that we get a functor 
$$
\Sm_\C^{op} \to \CAlg(\Sp), \qquad X \mapsto \map(\Sigma^{\infty}_+X, \bK).
$$
where we make use of the fact that every object of $\Sm_\C$ is a cocommutative coalgebra.
We want to compare {this} functor {with} the presheaf of $K$-theory spectra $X \mapsto \calK(X)$ that is constructed directly from the category of
vector bundles over $X$ with its tensor product via ring completion. The construction of this functor is for example done in \cite{buta2} and recalled in the proof below.

While the following result is certainly expected and is well-known in slightly different setups, we think it is new in the highly structured form stated here (since in particular one side of the equivalence only exists as a commutative algebra object through the machinery developed in Appendix \ref{qkdhkqwdlwdwqddqwdqd}) .
\begin{thm}
For any $X\in \Sm_\C$  we have an
equivalence
\[
\map(\Sigma^{\infty}_+X, \bK) \simeq \calK(X)
\]
in $\CAlg(\Sp)$ that is natural in $X$.
\end{thm}
 
\begin{proof}

{\bf Step 1:} 
 
\medskip
We abbreviate $$\spa:=\Fun^{Nis,\bbA^{1}}(\Sm_{\C}^{op},\Sp)\ , \quad 
    \Ring^{mot}_{S^{1}}:=\Fun^{Nis,\bbA^{1}}(\Sm_{\C}^{op},\CAlg(\Sp))\ .$$
   We first define the sheaf of algebraic $K$-theory spectra $\calK\in\Ring^{mot}_{S^{1}}$. 
We let $\Vect$ be the stack of vector bundles on $\Sm_{\C}$. It associates to $X\in \Sm_{\C}$ the groupoid of locally free and finitely generated $\mathcal{O}_{X}$-modules. This stack
has a semiring   structure  ({called rig structure in the following}) given by the direct sum and 
the tensor product of vector bundles. With this structure $\Vect$  can be considered as a sheaf of $\Rig$-categories:
$$\Vect\in \Fun^{Nis}(\Sm_{\C}^{op},\Rig(\mathbf{Cat}[W^{-1}]))\ .$$
We refer to    \cite{Gepner:2013aa} for this notion and details of the following constructions in the $\Rig$-context.
The nerve  provides a morphism $$\Nerve:\Rig(\mathbf{Cat}[W^{-1}])\to \Rig(\sSet[W^{-1}])\ .$$
 We have an adjunction
$$\Omega B :\Rig(\sSet[W^{-1}])\leftrightarrows \Ring(\sSet[W^{-1}]):incl\ , $$
where $\Omega B$ is the ring completion. 
It group-completes the underlying additive monoid of a rig.
 We finally have a morphism $$\mathrm{sp}: \Ring(\sSet[W^{-1}])\simeq \CAlg(\Sp^{\ge 0})\to  \CAlg(\Sp)\ .$$
{Using that the functors $L^{\A^{1}}$ and $L^{Nis}$ are symmetric monoidal, we define the}
 sheaf of algebraic $K$-theory spectra  by
$$\calK:=L^{\bbA^{1}}(L^{Nis}(\mathrm{sp}(\Omega B(\Vect))))\in {\Ring^{mot}_{S^{1}}}\ .$$
{(Since the schemes in $\Sm_{\C}$ are all regular, the sheaf of ring spectra $L^{Nis}(\mathrm{sp}(\Omega B(\Vect)))$ computes algebraic $K$-theory. Since this is $\bbA^{1}$-invariant for regular schemes, the application of $L^{\bbA^{1}}$ would not be necessary.)}

\bigskip

{\bf Step 2:} 

\medskip
We use the adjunctions
$$
\Sigma^{\infty}_{S^{1},+}: {\MotSpc} \leftrightarrows\spa : \Omega^{\infty}_{S^{1}}\ , \quad \Sigma_{\P^{1}}^{\infty}:\spa\leftrightarrows \MotSp:\Omega^{\infty}_{\P^{1}}\ .
$$ 
Since the left adjoints are   symmetric monoidal,  their right adjoints {admit canonical} lax symmetric monoidal {structures \cite[Cor.~7.3.2.7]{HA}}.

We consider the multiplicative group as a presheaf of categories $\tilde \G_{m}$ {on $\Sm_{\C}$ which sends any $X$ in $\Sm_{\C}$ to the category with one object and morphisms $\G_{m}(X)$.} 
Since $\G_{m}$ is commutative, these categories are symmetric monoidal in the natural way. We can interpret $\tilde \G_{m}$ as the prestack
of one-dimensional trivializable bundles. This provides a symmetric monoidal transformation $\tilde \G_{m}\to \Vect^{\otimes}$.   
Observing that $\Nerve(\tilde\G_{m}) \simeq B\G_{m}$ and using Lemma~\ref{lem:P1BGm}, we get a morphism 
$$
\P^{\infty}  \to L^{\bbA^{1}}{(L^{Nis}}(\Nerve( \Vect^{\otimes}))) \to \Omega_{S^{1}}^{\infty,\otimes}  \calK
$$
in $\Fun^{Nis,{\bbA^{1}}}(\Sm_{\C}^{op},\CAlg(\sSet[W^{-1}]))$, where we write $$\Omega_{S^{1}}^{\infty}\calK \in  \Fun^{ Nis,\bbA^{1}}(\Sm_{\C}^{op},\Ring(\sSet[W^{-1}]))$$ for the sheaf of ring spaces {underlying} $\calK$, and the superscript $\otimes$
 indicates that we keep the multiplicative monoid structure and forget the additive one.
 The adjoint gives  a map in $\Ring^{mot}_{S^{1}}$
 \begin{equation}\label{dkhdk2d32d3d23d23d3d}
\Sigma^{\infty}_{S^{1},+}\P^{\infty}\to \calK\ .
\end{equation}
Recall  the two morphisms $\P^{1}\to \P^{\infty}$ given by the standard inclusion and the inclusion of a base point.
They give rise to the $\spa$-versions $1_{S^{1}},\xi_{S^{1}}:\Sigma_{S^{1},+}^{\infty} \P^{1}\to \Sigma^{\infty}_{S^{1},+}\P^{\infty}$ of the maps $1,\xi$ in $\MotSp$ above. As above we define $\beta_{S^{1}}$ and get the map
$$\mu_{\beta_{S^{1}}}:\Sigma_{S^{1}}^{\infty}(\P^{1},\infty)\wedge \calK\to \calK\ .$$

\bigskip

{
\bf Step 3:} 

\medskip
As a consequence of the projective bundle formula we know that {the} adjoint  
$$
 {\calK \to \calK^{\Sigma_{S^{1}}^{\infty}(\P^{1},\infty)}}
$$
 of $\mu_{\beta_{S^{1}}}$  is  an equivalence.
The constructions in the proof of \cite[Thm.~6.1]{Naumann:2010aa}\footnote{The authors thank Markus Spitzweck for pointing out this argument.} now produce a spectrum $\bK^{\prime}\in \CAlg(\MotSp)$ together with a morphism in $\Ring^{mot}_{S^{1}}$ 
\begin{equation}\label{qwdqkjhdkjwqdwdklwqdwqdwqd}
\Sigma_{\P^{1}}^{\infty} \calK\to \bK^{\prime}
\end{equation} whose adjoint is an equivalence $  \calK \simeq \Omega_{\P^{1}}^{\infty} \bK^{\prime} $
in $\CAlg(\spa)$.
We then apply $\Sigma^{\infty}_{\P^{1}}$ to \eqref{dkhdk2d32d3d23d23d3d} and get a map in $\CAlg(\MotSp)$ 
 $$\Sigma^{\infty}_{+}\P^{\infty}\to \Sigma^{\infty}_{\P^{1}}\calK \xrightarrow{\eqref{qwdqkjhdkjwqdwdklwqdwqdwqd}} \bK^{\prime}\ .$$
 Since $\bK^{\prime}$ is Bott periodic {by construction}, this map naturally factors over 
 $$\bK\simeq \Sigma^{\infty}_{+}\P^{\infty}[\beta^{-1}]\to \bK^{\prime}$$  (we drop the forgetful functor \eqref{jkjelkqwjelkqwewqe}).
 By Theorem \ref{dlkqwkldqdwqdwqdwqd}  the induced map
 $$\Omega^{\infty}_{\P^{1}}\bK\to \Omega^{\infty}_{\P^{1}}\bK^{\prime}\simeq \calK$$
 is an equivalence. 
For every $X\in \Sm_{\C}$ we thus get equivalences in $\CAlg(\Sp)$
\begin{align*}
\map(\Sigma^{\infty}_{+}X, \bK) &\simeq \map(\Sigma^{\infty}_{\P^{1}}\Sigma^{\infty}_{S^{1},+} X, \bK) &&\text{(since $\Sigma^{\infty}_{\P^{1}}\Sigma^{\infty}_{S^{1},+} \simeq \Sigma^{\infty}_{+}$)} \\
&\simeq \map(\Sigma^{\infty}_{S^{1},+} X, \Omega^{\infty}_{\P^{1}}\bK) &&\text{(by Proposition \ref{proplaxad})}\\
&\simeq \map(\Sigma^{\infty}_{S^{1},+} X, \calK) &&\text{(by the above)} \\
&\simeq \calK(X). 
\end{align*}
\end{proof}

\section{The spectrum representing absolute Hodge cohomology}\label{sakbcsakjcascsacac8789}

{In this section we define the motivic absolute Hodge {cohomology} spectrum $\Hodge$ as a commutative algebra in the category of motivic spectra $\MotSp$. As a first step, we dualize the   logarithmic de Rham complex  functor } 
\[
A_{\log}\colon \Sm_{\C} \to \MHC_{ \R}^{op}. 
\] 
We will use the following notation{s} for {the} localization 
\begin{equation}\label{sqwswqsqwsqws}
\iota\colon\MHC_{\R}^{op} \to \MHC_{\R}[W^{-1}]^{op}
\end{equation} 
and the natural morphism to the {Ind}-completion ({see} Appendix \ref{lkjdqlkwdjlwqdwqdqdq})
 \begin{equation}\label{sqwswqsqwsqws1}\kappa\colon \MHC_{\R}[W^{-1}]^{op} \to \Ind(\MHC_{\R}[W^{-1}]^{op}).
\end{equation} 
We will write $\kappa$ only if it is really necessary to indicate that an object comes from $\MHC_{\R}[W^{-1}]^{op}$, for example if we want to dualize it.
 
The functor $A_{\log}$ is  lax symmetric monoidal.  Even better, its composition with the localization  $\iota A_{log}:=\iota \circ A_{\log}$ is symmetric monoidal  since the  morphisms
$$A_{\log}(X) \otimes A_{\log}(Y) \to A_{\log}(X \times Y)$$  
defining the lax symmetric monoidal structure
are quasi-isomorphisms.  The duality functor 
$$
(-)^\vee {= \ihom(-, 1)} : \MHC_{\R}^{op}[W^{-1}] \to \MHC_{\R}[W^{-1}]
$$
  is an equivalence of symmetric monoidal $\infty$-categories. To see this, we use the equivalence \eqref{ddghjkqwdhkjqwdqwd} between the $\infty$-category $\MHC_{\R}{[W^{-1}]}$
and the  $\infty$-derived  category of mixed  $\R$-Hodge {structures}. The latter
 has a perfect duality  since $\MHS_\R$ has one. 
We get a symmetric monoidal functor $(\iota A_{\log})^{\vee}$.

 \begin{proposition}\label{dqwidqpodwqdqwdqwdqwdwqd}
The functor $A_{\log}$ extends uniquely to a symmetric monoidal functor $\wA\colon \MotSp \to \Ind(\MHC_{\R}[W^{-1}])$ which preserves colimits such that 
\[
\xymatrix@C+2em{
\Sm_{\C} \ar[r]^-{A_{\log}} \ar[d] & \MHC_{\R}^{op} \ar[d]^{\kappa\circ(-)^\vee \circ \iota} \\
\MotSp \ar[r]^-{\wA} & \Ind(\MHC_{\R}[W^{-1}])
}
\]
commutes.
\end{proposition}
\begin{proof}
We want to use the universal property of the category of motivic spectra. First we factor the {clockwise} composition in the diagram through $\Sm_{ \C,\ast}$ 
using the fact that $\Ind(\MHC_{\R}[W^{-1}])$ is pointed.  In order to apply the universal property of motivic spectra $\MotSp$  {as} stated in Proposition \ref{univ},
 we use that
$\iota A_{\log}$ satisfies Nisnevich descent and is $\bbA^{1}$-invariant.  It remains to check
that   {$(\iota  A_{\log, *}(\P^1,\infty))^\vee$} is an invertible
object in $\Ind(\MHC_{\R}[W^{-1}])$. But indeed, we have an equivalence
\begin{equation}\label{oct2102}
(\iota A_{\log,*}(\P^1,\infty))^\vee \simeq \iota(\R(1)[2]) 
\end{equation}
in $\Ind(\MHC_{\R}[W^{-1}])$, where $ \R(1)[2]$ is the Tate $\R$-Hodge structure $\R(1)$ (see Example \ref{kldjldqwdqwdqwd}) viewed as mixed Hodge complex
concentrated in degree {$-2$}. Its inverse is $\iota(\R(-1)[-2])$.
\end{proof}

Now we use that $\MotSp$ and $\Ind(\MHC_{\R}[W^{-1}])$ are presentable. For the first $\infty$-category this is clear by construction and for the second it follows since it is the Ind-completion of an essentially small $\infty$-category
which admits all finite colimits. Since the functor $\wA$ preserves colimits, it admits a right adjoint $R$, i.e.~we have an adjunction \begin{equation}\label{qwdqwdqwdqwdqwdwqdwd}
 \wA :   \MotSp \leftrightarrows  \Ind(\MHC_{\R}[W^{-1}]) : R . 
\end{equation}

Since $\wA$ is symmetric monoidal, the right adjoint  {$R$} admits a canonical lax symmetric monoidal structure {\cite[Cor.~7.3.2.7]{HA}}.

\begin{remark}
 Morally one should think of the functor $\wA$ as assigning to every {motivic spectrum}  its mixed $\R$-Hodge {motivic spectrum} and {of} the functor $R$ as the 
forgetful functor from mixed  {$\R$-Hodge} {motivic spectra} into all {motivic spectra}.  
\end{remark}

\begin{lemma}\label{lem:R-preserves-colims}
The right adjoint $R$ in \eqref{qwdqwdqwdqwdqwdwqdwd} preserves all colimits.
\end{lemma}
\begin{proof}
Clearly, since everything is stable, $R$ preserves finite colimits. Thus it suffices to prove that it preserves filtered colimits. Since $\MotSp$ is compactly generated, this follows if we can show that 
$\wA$ preserves compact objects \cite[Prop.~5.5.7.2]{HTT}. By construction, $\wA$ sends suspension spectra of objects in $\Sm_{\C}$ to compact objects.  Again by construction, any object $X \in \MotSp$
can be written as a colimit $X \simeq \colim_{i \in I} X_i$ where each $X_i$ is a $\P^1$-shift of a suspension of some object
in $\Sm_{\C}$. In particular all $X_i$'s are sent to compact objects by $\wA$. If $X$ is compact, then we can factor the identity
$X \to X$ through $X_0 = \colim_{i \in I_0} X_i$ where $I_0 \subseteq I$ is finite. Certainly, $X_0$ gets mapped to a compact object since finite colimits of compact objects are compact. Thus $\wA(X)$ is a retract of 
the compact object $\wA(X_0)$ and therefore is itself compact.
\end{proof}

Recall the definition \eqref{ddehjwedgewjhdewdwed} of $\bbT\in \CAlg(\Ind(\MHC_{\R}[W^{-1}]))$. 
Its underlying object is 
\begin{equation}\label{lkswqsqwsqwswqsws} \bbT\simeq \bigoplus_{p\ge 0} \iota(\R(p)[2p])\in \Ind(\MHC_{\R}[W^{-1}]) \ .\end{equation} 
If we restrict the multiplication $\bbT\otimes \bbT\to \bbT$  to the inclusion 
$$\beta:\iota(\R(1)[2])\to \bbT$$
of the summand for $p=1$ in the left factor, then we get 
a map
 \begin{equation}\label{wqdwqddqdqdqd}
\nu_{\beta}:\iota(\R(1)[2]) \otimes \bbT \to \bbT\ . 
\end{equation} 
{Tensoring} this map with $\iota(\R(-1)[-2])$ and using the canonical equivalence 
$$\iota(\R(-1)[-2])\otimes  \iota(\R(1)[2])\simeq \iota(\R(0)[0])\simeq 1$$ we get a map 
\begin{equation}\label{frefrefrefrefrfrefre34r34r4r}
\mu_{\beta}\colon \bbT \to \iota(\R(-1)[-2])\otimes \bbT\ .
\end{equation}
By definition, $\bbT[\beta^{-1}] \in \CAlg(\Ind(\MHC_{\R}[W^{-1}] ))$ is obtained from $\bbT$ by inverting this map ({see} Appendix~\ref{localization} and {compare with} the construction of $\bK$).
\begin{definition}
We define the motivic absolute Hodge spectrum 
\[
\Hodge := R(\bbT[\beta^{-1}]) \in {\CAlg}(\MotSp)\ .
\]
\end{definition}
Since $R$ is lax symmetric monoidal, and since $\bbT[\beta^{-1}]$ is a commutative algebra in $\Ind(\MHC_{\R}[W^{-1}])$, 
 the motivic spectrum $\Hodge$ is indeed a  commutative algebra in $\MotSp$.
In view of Lemma~\ref{lem:R-preserves-colims} the decomposition \eqref{lkswqsqwsqwswqsws} induces a decomposition in $\MotSp$
\begin{equation}\label{ewdwedewdewdewdewdewd}
\Hodge\simeq \bigoplus_{p\in \Z} \Hodge(p)\ , \quad \Hodge(p)\simeq R(\R(p)[2p])\ .
\end{equation}
 
\begin{proposition}
For {any} $X$ in $\Sm_{\C}$ we have  a natural equivalence  
\[
{\map}(\Sigma^{\infty}_{+}X, \Hodge) \simeq H(\IDR(X))\ .
\]
in $\CAlg(\Sp)$.
\end{proposition}
In view of \eqref{eq:def-abs-Hodge} and Proposition~\ref{ednqwdkjweqdqwdqwd}, this settles Theorem~\ref{thm:Hodge-Intro}.
\begin{proof}
 By construction, we   have equivalences 
\begin{eqnarray*}
{\map}(\Sigma^{\infty}_{+}X, \Hodge) &\simeq& {\map}(\wA(\Sigma^{\infty}_{+}X), \bbT[\beta^{-1}]) \\
&\simeq& {\map}(\kappa(\iota A_{\log})^\vee(X), \bbT[\beta^{-1}]) \\
& \simeq &{\map}(1,\kappa(\iota A_{\log})(X) \otimes  \bbT[\beta^{-1}]).
\end{eqnarray*}
Furthermore,  by  Proposition \ref{ednqwdkjweqdqwdqwd} we have
 $$H(\IDR(X)) \simeq {\map}(1,\kappa(\iota A_{\log}(X))\ \otimes  \bbT)\ .$$ 
{It thus suffices to show that the} natural map  
\[
{\map} (1,\kappa(\iota A_{\log}(X))\otimes \bbT) \to {\map} (1,\kappa(\iota A_{\log}(X))\otimes  \bbT[\beta^{-1}])
\]
is an equivalence.
{But this  follows since $A_{\log}(X)$ has no non-trivial elements of negative weight and hence}
$${\map}(1, \kappa(\iota A_{\log}(X))\otimes \R(-i)[-2i]) $$
is contractible for positive $i$.
\end{proof}

\section{The regulator}

By definition of the Tate Hodge structure (see Example \ref{kldjldqwdqwdqwd}) we have canonical equivalences
$$\iota(\R(1)[2])\otimes  \iota(\R(p)[2p])  \simeq  \iota(\R(p+1)[2p+2])$$ 
for all $p\in \Z$.  Altogether they induce a  canonical equivalence 
\begin{equation}\label{2r3r2r32r3rredqwdq}
\iota(\R(1)[2])\otimes \bigoplus_{p\in \Z}  \iota(\R(p)[2p])\simeq \bigoplus_{p\in \Z}  \iota(\R(p)[2p])\ .
\end{equation}

\begin{proposition}\label{equivalence}
Depending on the choice of an identification
\begin{equation}\label{qwdqwdwqdqwddd}
\varepsilon:\wA(\Sigma^{\infty}(\P^{1},\infty))\simeq \iota(\R(1)[2])
\end{equation}
 in  $\Ind(\MHC_{\R}[W^{-1}])$
 we have an essentially unique equivalence
  \begin{equation}\label{r23r3r32r32r32r322er22423ewdqwdqwdd}
\wA(\bK) \overset{r_{\varepsilon}}{\simeq} \bbT[\beta^{-1}]\simeq \bigoplus_{p\in \Z}  \iota(\R(p)[2p])
\end{equation}
which sends the morphism $\nu_\beta:\Sigma^{\infty}(\P^{1},\infty)\wedge \bK\to \bK$ from \eqref{dewdewkjdnewkdedewdd}
 to the canonical identification \eqref{2r3r2r32r3rredqwdq}
 and identifies the unit $\Sigma^{\infty}_{+}\{*\}\to \bK$ with the canonical inclusion
 $\R(0)[0]\to \bbT[\beta^{-1}]$.
 \end{proposition}
 \begin{remark}\label{fewjfhjewkewfwefwefweffewfewfe} {Note that the
  additional conditions on the equivalence \eqref{r23r3r32r32r32r322er22423ewdqwdqwdd}
required in the statement of Proposition \ref{equivalence} are  equivalent to the condition 
  that the morphism  $r_{\varepsilon} $ induces a map of commutative algebra objects in 
$\Ho(\Ind(\MHC_{\R}[W^{-1}]))$.
} 
\end{remark}

\begin{proof}
By the construction of $\wA$ in Proposition \ref{dqwidqpodwqdqwdqwdqwdwqd} we have $\wA(\Sigma^{\infty}_+X)  \simeq
\kappa((\iota A_{\log}(X))^{{\vee}})$ for any smooth variety $X$, where the maps $\kappa$ and $\iota$ are  as in \eqref{sqwswqsqwsqws} and \eqref{sqwswqsqwsqws1}. 
{In the following, we again suppress $\kappa$ from the notation.}
Since $\wA$ preserves colimits,
we have
\begin{align}\label{eq:Alogcolim}
\wA(\bK) & \simeq \wA\left(\colim\left(\Sigma^{\infty}_+\P^{\infty}
\xrightarrow{\mu_\beta} \Sigma^{\infty}(\P^1,\infty)^{-1} \wedge \Sigma^{\infty}_+\P^{\infty} 
\xrightarrow{\id\wedge\mu_\beta} \dots \right) \right) \notag \\
& \simeq\colim\left( \wA(\Sigma^{\infty}_+\P^{\infty}) \xrightarrow{\mu_\beta }
\wA(\Sigma^{\infty}(\P^1,\infty)^{-1} \wedge \Sigma^{\infty}_+\P^{\infty})
\rightarrow \dots \right)\ ,
\end{align}
where $\mu_{\beta}$  is the map introduced   in \eqref{dqwdqwdkjqkjkhkqwdqwd}.
Since $\wA$ is symmetric monoidal, we have equivalences ({see} \eqref{oct2102}) 
\begin{align*}
\wA(\Sigma^{\infty}(\P^1,\infty)^{-1} \wedge \Sigma^{\infty}_+\P^{\infty}) & \simeq
\wA(\Sigma^{\infty}(\P^1,\infty))^{-1}\otimes \wA(\Sigma^{\infty}_+\P^{\infty}) \\
& \simeq \iota(\R(-1)[-2]) \otimes \wA(\Sigma^{\infty}_+\P^{\infty}).
\end{align*}
Furthermore, we have  equivalences
\[
\wA(\Sigma^{\infty}_+\P^{\infty})  \simeq \wA(\colim_n \Sigma^{\infty}_+\P^n)  \simeq \colim_n
\wA(\Sigma^{\infty}_+\P^n)  \simeq \colim_n \kappa(\iota(A_{\log}(\P^n))^{\vee})\ .
\]
We now use that
\[
\iota A_{\log} (\P^n) \simeq \bigoplus_{i=0}^n \iota(\R(-i)[-2i]) \ . 
\]
We can choose these identifications compatible with the restrictions along $\P^{n}\to \P^{n+1}$ for all $n$ and then get
 \begin{equation}\label{r34r23rr23r32r2weffwf}
\wA(\Sigma^{\infty}_+\P^{\infty}) \simeq \bigoplus_{i=0}^\infty \iota(\R(i)[2i])  \simeq \bbT\ .
\end{equation}
{This equivalence is essentially unique 
up to the action of $\pi_{0}(\Aut_{\Ind(\MHC_{\R}[W^{-1}])}(\bbT))$.}
By the calculation of mapping spaces in Example \ref{ijqwdiwqjdwqdljqwldkwqdqwdwqdwqd}   this 
group is given by
$$\pi_{0}(\Aut_{\Ind(\MHC_{\R}[W^{-1}])}(\bbT))\cong \prod_{p\in \nat} \R^{\times} $$
where the $p$-th factor corresponds to $\pi_{0}(\Aut_{\MHC_{\R}[W^{-1}]}(\iota(\R(p)[2p]))\cong \R^{\times}$.
At this point we fix the identification \eqref{r34r23rr23r32r2weffwf}  such that \begin{equation}\label{fwelkfjwelkfewfwfwf}
\xymatrix{\wA(\Sigma^{\infty}_+\{*\})\ar[r]^{can}_{\simeq}\ar[d]&\iota(\R(0)[0])\ar[d]\\ \wA(\Sigma^{\infty}_+\P^{\infty})\ar[r]&\bigoplus_{i=0}^\infty \iota(\R(i)[2i])  }
\end{equation}
commutes, where the left vertical map is induced by the inclusion of a point $*\to \P^{\infty}$, and the right vertical map is the canonical inclusion of the zeroth summand.

 We now fix once and for all the identification $\varepsilon$ that appears in \eqref{qwdqwdwqdqwddd}
\begin{equation}\label{32e32e32ekl32ejlk32e32e32}
\varepsilon\colon \wA (\Sigma^{\infty}(\P^1,\infty))\simeq \iota(\R(1)[2])\ .
\end{equation}
Since $\wA$ is symmetric monoidal, we then get an equivalence (depending on the choice of \eqref{r34r23rr23r32r2weffwf})
 \[
\wA(\Sigma^{\infty}(\P^1,\infty)^{-1} \wedge \Sigma^{\infty}_+\P^{\infty}) \simeq \iota(\R(-1)[-2]) \otimes
\bigoplus_{i=0}^{\infty} \R(i)[2i] \simeq \bigoplus_{i=-1}^{\infty} \R(i)[2i].
\]
Under {these} identification{s}, the transition map 
$\mu_\beta$ in \eqref{eq:Alogcolim}  {corresponds} to a morphism \begin{equation}\label{dqwkjndwqkldqdwqdqd}
 \bigoplus_{i=0}^{\infty} \R(i)[2i] \to \bigoplus_{i=-1}^{\infty} \R(i)[2i]\ .
\end{equation}
If we change
the equivalence \eqref{r34r23rr23r32r2weffwf} by an automorphism $(\lambda_{p})_{p\in \nat}\in \prod_{p\in \nat} \R^{\times}$ (with $\lambda_{0}=1$ in order to preserve \eqref{fwelkfjwelkfewfwfwf}), then this map is changed by the post-composition with
$$\frac{\lambda_{p+1}}{\lambda_{p}}\in\R^{\times}\cong  \pi_{0}(\Aut_{\MHC_{\R}[W^{-1}]}(\iota(\R(p)[2p])))$$ 
in the $p$-th component of the target for every $p\in \nat$. We see that there is an essentially  unique choice of \eqref{r34r23rr23r32r2weffwf} such that
\eqref{dqwkjndwqkldqdwqdqd} becomes the canonical embedding.
From now on we adopt this choice. We then get an equivalence \[
\wA(\bK) \overset{r_{\varepsilon}}{\simeq}\colim_{n\to\infty} \bigoplus_{i=-n}^{\infty} \R(i)[2i]
\]
 with the required properties.
\end{proof}

\begin{corollary}
The mapping space $\Map(\bK,\bH)$ is discrete and we have 
\[
\pi_0(\Map(\bK,\bH)) \cong \prod_{p\in\Z} \R.
\]
\end{corollary}
\begin{proof}
By the definition of $\bH$ and adjunction, we have
\[
\Map(\bK,\bH) \simeq \Map(\wA(\bK), \bbT[\beta^{-1}]).
\]
By Proposition~\ref{equivalence}, the right-hand side is equivalent to 
\[
\Map(\bbT[\beta^{-1}], \bbT[\beta^{-1}]) \simeq  {\prod_{p\in \Z} \colim_{n\to\infty} \prod_{q=-n}^{n}} \Map(\iota(\R(p)[2p]), \iota(\R(q)[2q])).
\]
By Example~\ref{ijqwdiwqjdwqdljqwldkwqdqwdwqdwqd}, only the factors with $p=q$ contribute and moreover 
\[
 \pi_{0}(\Map(\iota(\R(p)[2p]), \iota(\R(p)[2p]))) \cong \R
\]
is discrete.
\end{proof}

For later use we formulate the following corollary of the proof of Proposition \ref{equivalence}.
\begin{corollary}\label{dlqjwldqwdqdwqdwqdwwqd}
If we  post-compose   the equivalence {$\varepsilon$ from \eqref{qwdqwdwqdqwddd}} 
with
$$\lambda\in \R^{\times}\cong \pi_{0}(\Aut_{\MHC_{\R}[W^{-1}]}(\iota(\R(1)[2]))\ ,$$ 
then the equivalence  \eqref{r23r3r32r32r32r322er22423ewdqwdqwdd} changes by the post-composition {with} the equivalence
$$(\lambda^{p})_{p\in \Z}\in \prod_{p\in \Z} \R^{\times}\cong \pi_{0}(\Aut_{\MHC_{\R}[W^{-1}]}(\bbT[\beta^{-1}]))\ .$$
\end{corollary}

 \begin{definition} For every choice of the equivalence  $\varepsilon$ in   \eqref{qwdqwdwqdqwddd} 
 we define a morphism of motivic spectra $$\reg_{\varepsilon}\colon \bK\to \Hodge$$ by 
  applying the adjunction  \eqref{qwdqwdqwdqwdqwdwqdwd} to  
the equivalence $r_{\varepsilon}:\wA(\bK) \simeq \bbT[\beta^{-1}]$  given by \eqref{r23r3r32r32r32r322er22423ewdqwdqwdd} and using  the definition $\Hodge  := R( \bbT[\beta^{-1}])$. \end{definition} 
\begin{remark}\label{lkwjfwelkwejflfffewrwr}
The construction $\reg_{\varepsilon}\colon \bK\to \Hodge$ can be interpreted as the best approximation to $K$-theory in the world of mixed Hodge complexes. 

{By Remark \ref{fewjfhjewkewfwefwefweffewfewfe} the   morphisms $\reg_{\varepsilon}$ for all choices  of $\varepsilon$ in   \eqref{qwdqwdwqdqwddd} are precisely the morphisms $\bK\to \Hodge$
which induce maps of commutative algebras in the homotopy category of motivic spectra.} 
\end{remark}

\begin{thm}\label{lncakqjjhdkjqwdhqwdq}
The regulator map $\reg_{\varepsilon} \colon \bK \to \Hodge$ refines essentially uniquely to a morphism between objects of $\CAlg({\MotSp})$.
\end{thm}
\begin{proof} We must show that the fibre of
$${\Map}_{\CAlg(\MotSp)}(\bK,\Hodge)\to {\Map}_{\MotSp}(\bK,\Hodge)$$
at the point $\reg_{\varepsilon} \in {\Map}_{\MotSp}(\bK,\Hodge)$ is contractible. Using the adjunction \eqref{qwdqwdqwdqwdqwdwqdwd},
we can equivalently show that the fibre of
$$ {\Map}_{  \CAlg(\Ind(\MHC_{\R}[W^{-1}]))}(\wA(\bK), \bbT[\beta^{-1}])\to {\Map}_{ \Ind(\MHC_{\R}[W^{-1}])  }(\wA(\bK), \bbT[\beta^{-1}])$$ over the equivalence $r_{\varepsilon}:\wA(\bK)\simeq \bbT[\beta^{-1}]$ given in Proposition \ref{equivalence}
 is contractible. Since we consider the fibre over an equivalence, we can switch domain and target. It suffices to verify that the fibre of 
 \begin{equation}\label{dhkwekjdhwekdhkjewhdkjewdewdw}
{\Map}_{  \CAlg(\Ind(\MHC_{\R}[W^{-1}]))}(  \bbT[\beta^{-1}],\wA(\bK))\to {\Map}_{ \Ind(\MHC_{\R}[W^{-1}])  }(\bbT[\beta^{-1}],\wA(\bK))
\end{equation}
over the inverse equivalence $ r_{\varepsilon}^{-1}:\bbT[\beta^{-1}]\simeq \wA(\bK)$ is contractible.
Now we use the  universal property of the commutative algebra $\bbT[\beta^{-1}]$: the space of morphisms from $\bbT[\beta^{-1}]$
to any other commutative algebra object  is homotopy equivalent to the space of commutative algebra morphisms from $\bbT \to A$ which send the morphism $\beta$ in \eqref{wqdwqddqdqdqd} to an invertible morphism, see Appendix \ref{localization}.   We thus define
$${\Map}^{\beta^{-1}}_{  \CAlg(\Ind(\MHC_{\R}[W^{-1}]))}( \bbT ,\wA(\bK))\subseteq {\Map}_{  \CAlg(\Ind(\MHC_{\R}[W^{-1}]))}( \bbT ,\wA(\bK))$$
to be the union of {those} components {consisting} of maps which send $\beta$ to an invertible morphism.
Then we have an equivalence 
\begin{equation}\label{23r32r2r32r2r32rrdewdewdde}
{\Map}^{\beta^{-1}}_{  \CAlg(\Ind(\MHC_{\R}[W^{-1}]))}( \bbT  ,\wA(\bK))\simeq {\Map}_{  \CAlg(\Ind(\MHC_{\R}[W^{-1}]))}(  \bbT[\beta^{-1}],\wA(\bK))\ .
\end{equation}
 Since $\bbT$ is the free algebra {on} $\iota(\R(1)[2])$, restriction along
the canonical map $\iota(\R(1)[2])\to  \bbT$ induces an equivalence 
\begin{equation}\label{23r32r2r32r2r32rrdewdewdde1}
{\Map}_{  \CAlg(\Ind(\MHC_{\R}[W^{-1}]))}( \bbT ,\wA(\bK))\simeq {\Map}_{  \Ind(\MHC_{\R}[W^{-1}])}( \iota(\R(1)[2]),\wA(\bK))\ .
\end{equation}
Using the equivalence
$r_{\varepsilon}:\bbT[\beta^{-1}] \simeq \wA(\bK)$ and Example \ref{ijqwdiwqjdwqdljqwldkwqdqwdwqdwqd} we see that
the right-hand side is discrete  and given by
 $$ {\Map}_{\Ind(\MHC_{\R}[W^{-1}])  }(  \iota(\R(1)[2])  ,\bbT[\beta^{-1}]) \simeq  {\Map}_{\Ind(\MHC_{\R}[W^{-1}])  }(  \iota(\R(1)[2])  , \iota(\R(1)[2]) ) \cong \R\ .$$
 Under \eqref{23r32r2r32r2r32rrdewdewdde} and \eqref{23r32r2r32r2r32rrdewdewdde1}  we get the equivalence
 $$ {\Map}_{  \CAlg(\Ind(\MHC_{\R}[W^{-1}]))}( \bbT[\beta^{-1}] ,\wA(\bK))\simeq \R^{\times}\ .$$
Again using Example \ref{ijqwdiwqjdwqdljqwldkwqdqwdwqdwqd}  and the equivalence $r_{\varepsilon}:\bbT[\beta^{-1}] \simeq \wA(\bK)$ we see that the map
\eqref{dhkwekjdhwekdhkjewhdkjewdewdw} is equivalent to the map 
 \begin{equation}\label{skl1h2slk12s12s12s12s1}
\R^{\times}\to \prod_{p\in \Z}\R\ ,\quad \lambda \mapsto (\lambda^{p})_{p\in \Z}\ .
\end{equation}
Under these identifications $r_{\varepsilon}$ goes to $ r_{\varepsilon}\circ  r_{\varepsilon}^{-1}=\id_{\bbT[\beta^{-1}]}$, hence 
to $(1_{p})_{p\in \Z}$.
This point belongs to the image of \eqref{skl1h2slk12s12s12s12s1}. 
\end{proof}

\begin{lemma}\label{dklqwjdwqdwqdwqdwd}
There is a unique choice of the equivalence $\varepsilon$ in \eqref{qwdqwdwqdqwddd} 
such that the induced regulator map
 \begin{equation}\label{reg21}
\reg_{\varepsilon}:  K_{*}(X) \to \bigoplus_{p \in \mathbb{N}}  H^{2p-*}_{\absHodge}(X,\R(p))\ ,
\end{equation}
coincides with Beilinson's regulator.
\end{lemma}
\begin{proof}
We use \cite[Thm.~5.{6}]{Feliu:aa}. Due to the work of \cite{Burgos-Wang}
it is known that Beilinson's regulator is induced by a map between objects of $\Fun^{Zar}(\Sm_{\C},\sSet_{*}[W^{-1}])$ (this is realized by the model category $\mathbf{sT}_{*}$ in the reference)
such that the induced map in 
$\Ho(\Fun^{Zar}(\Sm_{\C},\sSet[W^{-1}]))$
 is a map of commutative monoids.
Similarly, the map 
$$\Omega^{\infty} {\map}(\Sigma^{\infty}_{+}(\dots),\bK)\to \Omega^{\infty} {\map}(\Sigma^{\infty}_{+}(\dots),\Hodge)$$
 induced by $\reg_{\varepsilon}$ can be considered as  such a map. Then by 
 \cite[Thm.~5.{6}]{Feliu:aa} both regulators coincide if they induce the same map
$$
K_{0}(Gr(N,k))\to \bigoplus_{p \in \mathbb{N}}  H^{2p}_{\absHodge}(Gr(N,k),\R(p))$$
{on $K_{0}$ of the Grassmannians $Gr(N,k)$} for all $N,k$. 
{If $$\phi=\sum_{p\in \nat} \phi_{p} :K_{0}(-)\to \bigoplus_{p\in \nat} H^{2p}_{\absHodge}(-,\R(p))$$ is a natural transformation of ring-valued functors, then
for a line bundle $L$  we necessarily have the relation
$$\phi_{p}([L])=\frac{1}{p!} \phi_{1}([L])^{p}\ .$$ This applies
to the Beilinson regulator as well as to  the transformation induced by \eqref{reg21} for $*=0$.} By the splitting principle
 it {now} suffices to show that they induce the same first Chern class. In detail, 
{the map \eqref{eq:1-xi} defines a map $\Sigma^{\infty}_{+}\P^{1}\to \bK$, i.e. a class $\ell\in \bK^{0}(\P^{1})\cong K_{0}(\P^{1})$.}  The restriction of $\reg_{\varepsilon}$ to this map 
 is the first Chern class 
  $$c_{\varepsilon} (\ell) \in  H^{2}_{\absHodge}(\P^{1},\R(1))\cong \R\ .$$

For every choice of $\varepsilon$ in {\eqref{qwdqwdwqdqwddd}} 
we have a multiplicative natural transformation
\eqref{reg21}.  
By Corollary \ref{dlqjwldqwdqdwqdwqdwwqd} we can rescale an initial choice of $\varepsilon$ in a unique way  
 such that afterwards 
 $c_{\varepsilon}(\ell)$ coincides with the first Chern class leading to Beilinson's regulator.
\end{proof}

We write $\reg:=\reg_{\varepsilon}$ for the choice of the equivalence {$\varepsilon$ in} {\eqref{qwdqwdwqdqwddd}} fixed in Lemma  \ref{dklqwjdwqdwqdwqdwd}.

\begin{proof}[Proof of Theorem~\ref{thm:main-thm-intro}]
By Lemma \ref{dklqwjdwqdwqdwqdwd} the morphism  of motivic spectra   $\reg:\bK\to \Hodge$  represents Beilinson's regulator. By Remark \ref{lkwjfwelkwejflfffewrwr} it induces a morphism of commutative algebras in the homotopy category. Combining this remark and Theorem~\ref{lncakqjjhdkjqwdhqwdq} every such morphism refines essentially uniquely   to a morphism in $\CAlg(\Sp^{\P^{1}})$. 
\end{proof}

For any smooth variety $X\in \Sm_{\C}$, using the diagonal, the motivic spectrum $\Sigma^{\infty}_{+}X$ refines naturally to a cocommutative coalgebra. 
Thus, by Corollary~\ref{cor:map-calg}, the algebraic $K$-theory spectrum of $X$ defined by $$\bK(X):= \map(\Sigma^{\infty}_{+}X, \bK)$$ refines naturally to a commutative ring spectrum. 
\begin{corollary}
For any $X\in\Sm_{\C}$, the Beilinson regulator refines to a morphism of commutative ring spectra
\[
\reg_{X}\colon \bK(X) \to H(\IDR(X))
\]
in a way which is natural in $X$.
\end{corollary}
This finally proves Theorem~\ref{firstthm}.
\begin{proof}
Applying $\map(\Sigma^{\infty}_{+}X, -)$ to $\reg:\bK\to \Hodge$ this follows from Theorem  \ref{lncakqjjhdkjqwdhqwdq}.
\end{proof}

\appendix

\section{Infinity categories and weak equivalences}\label{appendix_Kram}
\label{lkjdqlkwdjlwqdwqdqdq}

Throughout this paper we freely use the language of $\infty$-categories as developed by Joyal, Lurie and many others. An $\infty$-category is an inner Kan simplicial set $\bC$. 
Every ordinary category $\bC$ gives rise to an $\infty$-category by taking the nerve $N\bC$. This embeds ordinary categories into $\infty$-categories and we will usually drop the nerve notation and identify the $1$-category
$\bC$ with the associated $\infty$-category. For a simplicially enriched category $\bC^\Delta$ there is also a variant of the nerve (the homotopy coherent nerve) $\N\bC^\Delta$ which is an $\infty$-category if $\bC^\Delta$ is enriched over Kan complexes 
(see \cite[Section 1.1.5]{HTT}). 

For example, we let $\Spc^\Delta$
be the simplicial category whose objects are Kan complexes and whose 
simplicial sets of homomorphisms are the internal Hom-objects of simplicial sets. These are Kan complexes as well and thus we obtain 
an $\infty$-category which we denote as $\Spc := \N(\Spc^\Delta)$ and call it the $\infty$-category of spaces. Another instance is the simplicially enriched category $\Cat_\infty^\Delta$ of $\infty$-categories. 
The simplicial set of morphisms
between $X, Y$ is given by the maximal Kan complex contained in the simplicial set $\Hom(X,Y)$. Then the $\infty$-category of $\infty$-categories is defined to be $\Cat_\infty := \N(\Cat_\infty^\Delta)$.
Note that the objects in $\Cat_\infty$ can be large $\infty$-categories (as opposed to small). Thus $\Cat_\infty$ will be very large, which just means that we have to assume three nested Grothendieck universes (small, large, very large) in which we work.
 
There is another way of constructing $\infty$-categories. Therefore consider an $\infty$-category $\bC$ (which will in practice very often be a 1-category) equipped with a subset $W$ of the set of edges.
 We will call the elements of $W$
weak equivalences. Then an $\infty$-category $\bD$ together with a functor $i: \bC \to \bD$ is called {\emph{(Dwyer-Kan) localization of $\bC$ with respect to $W$}} if the functor $i$ sends weak equivalences in $\bC$ to equivalences in $\bD$
and $\bD$ satisfies the following universal property: for every $\infty$-category $\bE$ the functor
$$ i^*: \Fun(\bD,\bE) \to \Fun(\bC, \bE) $$
is fully faithful and the essential image consists of those functors $\bC \to \bE$ which carry weak equivalences in $\bC$ to equivalences in $\bE$. 
It is clear that $\bD$ is essentially uniquely determined by this universal property and it can be shown that a relative nerve exists for every $\infty$-category $\bC$ with weak equivalences.

 Since we will use the term `essentially unique' a number of times in this paper, let us spell out explicitly what this means here: consider the full subcategory $\bA$ of the slice $\infty$-category 
$(\Cat_\infty)_{\bC/}$ which consists of all relative nerves of $\bC$ with respect to $W$. 
 Then $\bA$ is a contractible Kan complex. 
  There are several explicit constructions in the literature for relative nerves, for example Dwyer-Kan's Hammock localization or fibrant replacements in the marked model structure.
All of these constructions are necessarily equivalent by the above uniqueness assertion. We once and for all fix one explicit construction for the localization of $\bC$ and denote it by $\bC[W^{-1}].$ 
We assume that the construction is functorial in $\bC$ (i.e. in functors that preserve weak equivalences).

\begin{example}
The canonical inclusion $\sSet \to \Spc$ exhibits $\Spc$ as the localization at the class of weak equivalences. Here $\sSet$ is the 1-category of simplicial sets. Thus we have $\Spc \simeq \sSet[W^{-1}]$.
More generally, for a model category $\bM$ we can form $\bM[W^{-1}]$ and this is an enhancement of the homotopy category $\Ho(\bM)$ which is obtained by universally inverting $W$ in the world of 1-categories. 
If $\bM$ admits a simplicial enrichment $\bM^\Delta$ making it a simplicial model category then the $\infty$-category $\bM[W^{-1}]$ (which only depends on the underlying 1-category $\bM$ and the notion of weak equivalence) 
is equivalent to the homotopy coherent nerve of the full 
simplicial subcategory $\bM^\Delta_{cf} \subset \bM^\Delta$ on the fibrant and cofibrant objects (which a priori depends on the simplicial mapping spaces).
\end{example}

A lot of theory has been developed for $\infty$-categories which parallels well known results and concepts in ordinary category theory. Many of the results listed below are due to Joyal \cite{Joyal}, but our main sources will be the books of Lurie \cite{HTT} and \cite{HA}. Let us list the aspects that we will need in this paper:
\begin{enumerate}
 \item There is a notion of limit and colimit in an $\infty$-category \cite[Section 1.2.13]{HTT}. This generalizes the notion of homotopy limit and colimit in model categories as shown in \cite[Section 4.2.4]{HTT}. 
The properties of (co)limits in ordinary categories mostly carry over to that world \cite[Chapter 4]{HTT}, for example pullbacks and pushouts satisfy a pasting-law of which we will make repeatedly use in this paper. 
Important for us is that there is a notion of a filtered $\infty$-category and filtered (co)limits \cite[5.3.1]{HTT}.
 \item There is the notion of adjoint functors \cite[5.2]{HTT}, and {it} behaves similarly to the ordinary case. The most important fact for us is that left adjoint functors preserve all colimits and right adjoints preserve all limits. 
Adjunctions between model categories give rise
to adjunctions between the associated $\infty$-categories \cite{Mazel-Gee}. 

\item For every pair of $\infty$-categories {$\bC,\bD$,} the internal Hom in simplicial sets is again an $\infty$-category and denoted $\Fun(\bC,\bD)$. The functor category $\Fun(\bC^{op}, \Spc)$ is called the presheaf category on 
$\bC$ and denoted as $\mathcal{P}(\bC)$. The $\infty$-categorical Yoneda embedding defines a fully faithful inclusion $\bC \to \mathcal{P}(\bC)$ sending $c \in \bC$ to the functor 
$\Map_\bC(-,c)$ \cite[5.1.3]{HTT}.
 \item There is a notion of a presentable $\infty$-category which combines a set-theoretical smallness condition (accessibility) with the existence of all colimits and limits \cite[5.5]{HTT}. The notion of being accessible is closely related to the 
ind-completion that we will discuss below. The underlying $\infty$-category of a combinatorial, simplicial model category is presentable, and in fact every presentable $\infty$-category arises in that way. 
The class of presentable $\infty$-categories has good closure properties, for example, Bousfield localizations of presentable $\infty$-categories are usually again presentable (precisely: if the localization is accessible 
\cite[Rem.~5.5.1.6]{HTT}). 
Note that the analogue of presentability in ordinary category theory is usually called `locally presentable'. 

In the presentable setting we have the 
adjoint functor theorem in its cleanest form, namely that a functor between presentable $\infty$-categories is left adjoint precisely if it preserves all small colimits \cite[Cor.~5.5.2.9]{HTT}.

\item There are the notions of a symmetric monoidal $\infty$-category and of (lax) symmetric monoidal functors \cite[Chapter 2]{HA}. These are compatible with the constructions of $\infty$-categories outlined above. If a simplicial category 
$\bC^\Delta$ has a symmetric monoidal structure, then the nerve $\N \bC^\Delta$ inherits a natural symmetric monoidal structure as well. For a 1-category $\bC$, a symmetric monoidal structure on the associated 
$\infty$-category in the $\infty$-categorical sense is (essentially) equivalent to a symmetric monoidal structure on $\bC$. If $\bC$ admits a symmetric monoidal structure and a class of weak equivalences, such that
the tensor product is homotopical in both variables (i.e.~preserves weak equivalences), then there is an induced symmetric monoidal structure on $\bC[W^{-1}]$.

Note that we  slightly deviated from Lurie's terminology  here and write a symmetric monoidal category only as the underlying $\infty$-category $\bC$ leaving the tensor product implicit and not as $\bC^\otimes$. 
Also the terminology of lax symmetric monoidal functors has not been used by Lurie explicitly. A lax symmetric monoidal functor $\bC \to \bC'$ is just a map of the underlying $\infty$-operads in the terminology of \cite{HA}.

In a symmetric monoidal $\infty$-category $\bC$ we can define commutative algebra objects and an $\infty$-category {of commutative algebra objects} $\CAlg(\bC)$. Every lax symmetric  monoidal functor $f: \bC \to \bD$ induces a functor $\CAlg(\bC) \to \CAlg(\bD)$ which on underlying 
objects sends $c$ to $f(c)$.

\item \label{pressymmon} We call a symmetric monoidal $\infty$-category presentably symmetric monoidal if its underlying $\infty$-category is presentable and the tensor product preserves colimits in both variables separately. For every
simplicial, combinatorial, symmetric monoidal model category the underlying symmetric monoidal $\infty$-category has this property \cite[Rem.~4.1.3.10]{HA}. The converse is even true, namely that every presentably symmetric monoidal $\infty$-category arises in
this way from a symmetric monoidal model category \cite{NikSag}. 

If $\bC$ is presentably symmetric monoidal, then the $\infty$-category of commutative algebras $\CAlg(\bC)$ is also presentable, in particular, it has all limits and colimits.
\item \label{stabilization}
An $\infty$-category $\bC$ is called stable if it admits finite limits and colimits, is pointed (i.e.~has an object that is initial and terminal), and a square $\Delta^1 \times \Delta^1 \to \bC$ depicted as
$$
\xymatrix{
A \ar[r]\ar[d] & B \ar[d] \\
C \ar[r] & D
}
$$
(we omit the diagonal map and the filling homotopies) is a pushout if and only if it is a pullback (see Chapter 1 of \cite{HA} for a discussion). Examples are given by the $\infty$-category $\Sp$ of spectra and, for every abelian category $\bA$, the $\infty$-category $\Ch(\bA)[W^{-1}]$ where $W$ is the class of quasi-isomorphisms.

For every $\infty$-category $\bC$ which admits finite limits there is a stabilization $\Sp(\bC)$, see \cite[1.4.2]{HA}. This in turn is functorial, i.e.~for every functor $\bC \to \bD$ which preserves finite limits we get an 
induced functor $\Sp(\bC) \to \Sp(\bD)$. For example, $\Sp(\Spc) \simeq \Sp$ and $\Sp(\Fun(\bC,\Spc)) \simeq \Fun(\bC, \Sp)$.
\end{enumerate}
Now we briefly recall the construction of the Ind-completion of an $\infty$-category which will be essential for this paper, see also \cite[Section 5.3]{HTT}. Let $\bC$ be an $\infty$-category. Then the Ind-completion is obtained by formally 
adding filtered colimits to $\bC$. 

\begin{definition}\label{def:Ind-completion}
We say that a functor $i: \bC \to \bD$ exhibits $\bD$ as the Ind-completion of $\bC$ if $\bD$ has all filtered colimits and if for every $\infty$-category $\bE$ which has all filtered colimits the induced morphism
$$
i^*: \Fun^\omega(\bD , \bE) \to \Fun(\bC,\bE)
$$
is an equivalence. Here $\Fun^\omega(\bD,\bE)$ denotes the full subcategory of $\Fun(\bD,\bE)$ spanned by the functors  which preserve filtered colimits.
\end{definition} 
An explicit model for the Ind-completion of $\bC$ can be obtained as follows. Let $\mathcal{P}(\bC)$ be the $\infty$-category of space-valued presheaves on $\bC$. Then the Yoneda embedding defines a fully faithful 
inclusion $\bC \to \mathcal{P}(\bC)$ and the objects in the essential image are called representable. The category $\mathcal{P}(\bC)$ has all colimits (in fact it is the universal category obtained from $\bC$ 
by adding all colimits).
We let $\Ind(\bC)$ be the full subcategory of $\mathcal{P}(\bC)$ which contains the representables and which is closed under filtered colimits. Then the Yoneda embedding exhibits $\Ind(\bC)$ as the Ind-completion
of $\bC$. 
In particular, this shows that the morphism $\bC \to \Ind(\bC)$ is fully faithful.

Let us now list a few properties of the Ind-completion that we frequently use in the paper:
\begin{itemize}
\item If $\bC$ is small and has all finite colimits, then $\Ind(\bC)$ is presentable, in particular it has all colimits. In this case the functor $\bC \to \Ind(\bC)$ preserves finite colimits \cite[Thm.~5.5.1.1]{HTT}.
\item If $\bC$ is symmetric monoidal, then $\Ind(\bC)$  also admits a symmetric monoidal structure, namely the unique extension of $\otimes$ that preserves filtered colimits in both variables separately 
\cite[Cor.~4.8.1.13]{HA}. 
If $\bC$ admits all small colimits and the tensor product preserves small colimits separately in each variable, then $\Ind(\bC)$ is presentably symmetric monoidal 
(see \ref{pressymmon} for this notion). 
\item
If $\bC$ is stable,  then so is $\Ind(\bC)$ \cite[Prop.~1.1.3.6]{HA}. 
\item If $\bC$ is (the nerve of) a 1-category, then $\Ind(\bC)$ also is a 1-category. This implies in particular that it agrees with the classical Ind-completion, i.e.~$N \Ind(\bC) \simeq \Ind(N\bC)$. 
\item If $\bC$ is an $\infty$-category, and if $W$ a class of 1-morphisms in $\bC$, then we call a morphism in $\Ind(\bC)$ a weak equivalence, if its image in $\Ind(\bC[W^{-1}])$ is an equivalence. We get a canonical morphism
$$(\Ind(\bC)) [W^{-1}] \to \Ind(\bC[W^{-1}]).$$
\item
If a functor $F: \bC \to \bD$ has a symmetric monoidal structure then $\Ind(F): \Ind(\bC) \to \Ind(\bD)$ also inherits a symmetric monoidal structure. In particular we get that
$\Ind(\bC)\to \Ind(\bC[W^{-1}])$ is symmetric monoidal if $\bC$ admits a tensor product that is homotopical with respect to $W$.
\end{itemize}

\section{Representable functors and algebra structures}\label{qkdhkqwdlwdwqddqwdqd}

Let $\bC$ be an $\infty$-category. Then for every $c \in \bC$ we have the representable functor 
$$
\Map(-,c): \bC^{op} \to \Spc.
$$
If $\bC$ is presentable, then an abstract functor $F: \bC^{op} \to \Spc$ is representable precisely if it preserves limits (i.e.~it sends colimits in $\bC$ to limits in the $\infty$-category of spaces). 
If $\bC$ is a stable $\infty$-category, then
for every $c \in \bC$ the representable functor refines to a functor
$$
\map(-,c): \bC^{op} \to \Sp .
$$ 
The functor $\map(-,c)$ has the property that it sends colimits in $\bC$ to limits of spectra and that 
\begin{equation}\label{eqn_eins-d}
\Map(-,c) \simeq \Omega^\infty \map(-,c)\ .
\end{equation}
These two properties determine $\map(-,c)$ essentially uniquely. If $\bC$ is presentable and stable, then the converse is also true: an abstract functor is of the form $\map(-,c)$ precisely if it
sends colimits to limits.

For every stable $\bC$ the assignment $c \mapsto \map(-,c)$ gives a functorial assignment $\bC \to \Fun(\bC^{op}, \Sp)$ which is fully faithful. We will refer to this functor as the stable 
Yoneda embedding. A concrete description of the stable Yoneda embedding as a functor can be given by applying the stabilization functor $\Sp(-)$ (see \ref{stabilization}) to the limit preserving Yoneda embedding to obtain 
$$
\bC \simeq \Sp(\bC) \to \Sp(\Fun(\bC^{op},\Spc)) \simeq \Fun(\bC^{op},\Sp).
$$

Now assume that $\bC$ admits a symmetric monoidal structure (here $\bC$ is not necessarily stable). Then also $\bC^{op}$ admits a symmetric monoidal structure which is constructed by first straightening the associated coCartesian fibration, 
then taking the opposite object-wise and unstraightening back to a coCartesian fibration. There is also a more explicit model discussed in \cite{BarDualizing}. Informally speaking this new symmetric monoidal structure 
is given by the `same' tensor product since the objects of $\bC^{op}$ are the same as those of $\bC$ but the coherence structure has to be modified accordingly. 
Then the presheaf category $\mathcal{P}(\bC) = \Fun(\bC^{op}, \Spc)$ admits a Day convolution 
type symmetric monoidal structure as introduced by Glasman \cite{Glas} generalizing the classical Day convolution \cite{Day}. Let us state the more general result that is shown there:

\begin{proposition}[Glasman]\label{prop_glas}
If $\bC$ is symmetric monoidal and $\bD$ is presentably symmetric monoidal,  then the functor category $\Fun(\bC,\bD)$ admits a symmetric monoidal structure with the following properties:
\begin{enumerate}
 \item For two functors $F,G: \bC \to \bD$, the tensor product $F \otimes G \in \Fun(\bC,\bD)$ is equivalent to the left Kan extension of the functor
$$
\bC \times \bC \xto{F \times G} \bD \times \bD \xto{\otimes_{\bD}} \bD
$$
along the functor $\otimes_{\bC}: \bC \times \bC \to \bC$ \cite[Lemma 2.3, Prop.~2.9]{Glas}.
\item
A commutative algebra structure on a functor $F: \bC \to \bD$ with respect to the  Day convolution on $\Fun(\bC, \bD)$ is equivalent to a lax symmetric monoidal structure on $F$ \cite[Prop.~2.10]{Glas}. More precisely 
there is an equivalence of $\infty$-categories
$$
\CAlg(\Fun(\bC,\bD)) \simeq \Fun^{lax}(\bC,\bD).
$$
\item\label{point3}
For every left adjoint, symmetric monoidal functor $\bD \to \bD'$ the induced functor $\Fun(\bC,\bD) \to \Fun(\bC,\bD{'})$ admits a symmetric monoidal refinement. 
\item 
The Yoneda embedding $\bC \to \Fun(\bC^{op}, \Spc)$ admits a symmetric monoidal refinement, where the target is equipped with the Day convolution structure \cite[Section~3]{Glas}.
\end{enumerate}
\end{proposition}

\begin{corollary} \label{cor_yoneda}
For a given object $c \in \bC$ in a symmetric monoidal $\infty$-category $\bC$, there is a homotopy equivalence between the space of lax symmetric monoidal structures on $\Map(-,c): \bC^{op} \to \Spc$ and the space of refinements of $c$ to an object of $\CAlg(\bC)$.
\end{corollary}
\begin{proof}
The space  $\CAlg_\bC(c)$ of commutative algebra structures on $c \in \bC$ is defined as the pullback
$$
\xymatrix{
\CAlg_\bC(c) \ar[d]\ar[r] & \CAlg(\bC) \ar[d] \\
\Delta^0 \ar[r]^{c} & \bC
}
$$
in $\Cat_\infty$ where the right vertical functor is the forgetful functor. This functor reflects equivalences, which implies that $\CAlg_\bC(c)$ is an $\infty$-groupoid.  

We now use the fact that the Yoneda embedding is fully faithful and symmetric monoidal (as stated in Proposition \ref{prop_glas}) to conclude that the induced functor
$$\CAlg(\bC) \to \CAlg(\mathcal{P}(\bC))$$ is also fully faithful and thus the resulting diagram
$$
\xymatrix{
\CAlg(\bC) \ar[d]\ar[r] & \CAlg(\mathcal{P}(\bC))\ar[d] \\
\bC \ar[r] & \mathcal{P}(\bC)
}
$$
is a pullback diagram. By pasting together the two pullback diagrams we conclude that we have an equivalence
$$
\CAlg_\bC(c) \simeq  \CAlg_{\mathcal{P}(\bC)}(c)
$$
i.e.~that the space of algebra structures on $c$ is equivalent to the space of algebra structures on $\Map(-,c)$. But the latter is equivalent to the space of lax monoidal structures on that functor 
(as stated in Proposition \ref{prop_glas}) which 
finishes the proof.
\end{proof}

Now assume that $\bC$ is stably symmetric monoidal and $c \in \bC$. By stably symmetric monoidal we mean that the tensor bifunctor $\otimes: \bC \times \bC \to \bC$ preserves finite colimits in both variables separately 
(or, equivalently, finite limits in both variables separately). 

\begin{proposition}
The stable Yoneda embedding 
\[
\bC \to \Fun(\bC^{op}, \Sp)
\]
admits a lax symmetric monoidal refinement. 
\end{proposition}
\begin{proof}
We assume that $\bC$ is a small, stable $\infty$-category. Since the Yoneda embedding $\bC \to \Fun(\bC^{op}, \Spc)$ admits a symmetric monoidal refinement we conclude that the resulting tensor product agrees with the one 
constructed by Lurie in \cite[Cor.~4.8.1.12]{HA}. Thus is also inherits the universal property of \cite[Prop.~4.8.1.10]{HA}, namely that colimit preserving symmetric monoidal functors $\Fun(\bC^{op}, \Spc) \to \bD$ are essentially the same thing as
symmetric monoidal functors $\bC \to \bD$ where $\bD$ is presentably symmetric monoidal. In particular we get an induced symmetric monoidal functor 
$$\Fun(\bC^{op}, \Spc) \to \Ind(\bC)$$
induced from the inclusion $\bC \to \Ind(\bC)$. The right adjoint of that functor is the canonical inclusion $\Ind(\bC) \to \Fun(\bC^{op}, \Spc)$ which extends the Yoneda embedding. Invoking \cite[Cor.~7.3.2.7]{HA} 
this right adjoint 
inherits a lax symmetric monoidal structure.  
The symmetric monoidal functor 
$$
\Sigma^\infty_+: \Fun(\bC^{op},\Spc) \to \Fun(\bC^{op},\Sp)
$$
exhibits the $\infty$-category $\Fun(\bC^{op},\Sp)$ as the stabilization of $\Fun(\bC^{op},\Spc)$ (as a symmetric monoidal $\infty$-category). Since $\Ind(\bC)$ is stable (see Appendix \ref{lkjdqlkwdjlwqdwqdqdq}), we get an induced left adjoint, symmetric monoidal functor 
$
\Fun(\bC^{op}, \Sp) \to \Ind(\bC)
$
that makes the diagram 
$$
\xymatrix{
\Fun(\bC^{op}, \Spc) \ar[d]^{\Sigma^\infty_+} \ar[r] & \Ind(\bC) \\
\Fun(\bC^{op}, \Sp) \ar[ru] & 
}
$$
commutative. But then the right adjoint $\Ind(\bC) \to \Fun(\bC^{op}, \Sp)$ inherits a canonical lax symmetric monoidal structure. Looking at the right adjoints in the commutative diagram above we see 
that the right adjoint functor of $\Fun(\bC^{op}, \Sp) \to \Ind(\bC)$ induces the stable Yoneda embedding. 

If finally $\bC$ is not small then we pass to a higher universe in which it becomes small.
\end{proof}

\begin{corollary} \phantomsection \label{cor:map-calg}
\begin{itemize}
 \item 
For any $c \in \CAlg(\bC)$, the functor $\map(-,c)$ has a refinement to an object of $\CAlg(\Fun(\bC^{op},\Sp)) \simeq \Fun^{lax}(\bC^{op}, \Sp)$. In particular, it sends a coalgebra $b\in \CAlg(\bC^{op})$ to an algebra
\[
\map(b,c) \in \CAlg(\Sp).
\]
\item
For any coalgebra  $b \in \CAlg(\bC^{op})$ the functor $\map(b,-): \bC \to \Sp$ has a refinement to a lax symmetric monoidal functor.
\end{itemize}
\end{corollary}

\begin{proposition}\label{proplaxad}
Consider an adjunction 
$$ L: \bC \leftrightarrows \bD : R $$
between stably symmetric monoidal $\infty$-categories with $L$ symmetric monoidal (and $R$ lax symmetric monoidal accordingly). Then  we have an equivalence of lax symmetric monoidal functors
$$
\map(L(-), c) \simeq \map(-, Rc) : \bC^{op} \to \Sp
$$
for every commutative algebra $c \in \CAlg(\bD)$.
\end{proposition}
\begin{proof}
From the way the lax symmetric monoidal structures are constructed we see that it is enough to prove that the following square
$$
\xymatrix{ 
\bD \ar[r]\ar[d]^R & \Fun(\bD^{op},\Sp) \ar[d]^{L^*} \\
\bC \ar[r] & \Fun(\bC^{op},\Sp) 
}
$$
commutes as a square of lax symmetric monoidal functors. We factor this square as 
$$
\xymatrix{ 
\bD \ar[r]\ar[d]^R &\Ind(\bD) \ar[r] \ar[d]^{\Ind(R)} & \Fun(\bD^{op},\Sp) \ar[d]^{L^*} \\
\bC \ar[r]&\Ind(\bC) \ar[r]  & \Fun(\bC^{op},\Sp) 
}
$$
and prove that the right hand square commutes (since the left hand does by construction). For this it is enough to consider the adjoint diagram which  sits on the right in a diagram
$$
\xymatrix{ 
\bC \ar[r]\ar[d]^{L} & \Fun(\bC^{op},\Spc) \ar[r]\ar[d]^{L_{!}} &\Fun(\bC^{op},\Sp) \ar[r] \ar[d]^{L_!} & \Ind(\bC) \ar[d]^{L} \\
\bD \ar[r] & \Fun(\bD^{op},\Spc) \ar[r] &\Fun(\bD^{op},\Sp) \ar[r]  & \Ind(\bD)
}
$$
where the left two squares commute (as symmetric monoidal functors). 
By the universal property of the stabilization and of the presheaf category, the right hand square commutes if the outer square commutes. But this is true by construction.
\end{proof}

\section{Algebra Localizations}\label{localization}

 Let $R$ be a commutative algebra in a symmetric monoidal $\infty$-category $\bC$, and let $\beta: I \to R$ be a morphism in $\bC$ where we assume that $I$ is a tensor invertible object of $\bC$. 
One example to have in mind is that $R$ is a ring spectrum and $\beta$ is an element in
a homotopy group of $R$. We want to describe the localization $R[\beta^{-1}]$ and explain why it is again a commutative algebra object. 
A priori, one has to be
careful about the distinction between algebra 
and module localization, though in the end they turn out to be the same. For what follows we assume that $\bC$ admits all colimits and that the tensor product preserves colimits in both variables separately. 

Let us first explain the localization as a module. We denote by $\Mod(R)$ the $\infty$-category of left $R$-modules (which turns out to be equivalent to the category of right $R$-modules \cite[4.5.1]{HA}). 
Let $M$ be a left $R$-module. The map $\beta$ induces the map
$$
\nu_\beta: I \otimes M \xrightarrow{\beta \cdot -} M 
$$
where the tensor product is taken in $\bC$ and not over $R$. We say that $\beta$ acts invertibly on $M$ if this morphism is an equivalence. 
Since $I$ is tensor invertible and $\bC$ is symmetric monoidal, we can as well view $\nu_\beta$ as a morphism $\mu_\beta: M \to I^{-1}\otimes M$. Then $\nu_\beta$ is an equivalence precisely if $\mu_\beta$ is an equivalence.

In order to construct the localization we will need the following well known cyclic invariance criterion which is discussed for example in \cite{Robalo} in the case that $\bC$ is the category of symmetric monoidal $\infty$-categories.\footnote{We would like to thank Markus Spitzweck and David Gepner for helpful discussions concerning the cyclic invariance condition.}
The morphism $\beta: I \to R$ leads to a new morphism $\beta^3: I^3 \to R$ which carries an action of the 
permutation group $\Sigma_3$. More precisely we have a map
$B \Sigma_3 \to \Map(I^3, R)$ sending the basepoint to $\beta^3$. 
\begin{definition}\label{def:cyclic-invariant}
The element $\beta$ is said to be cyclically invariant if the endomorphism induced by the cyclic permutation $\sigma = (123) \in \Sigma_3$ on $\beta^3 \in \Map(I^3,R)$ is homotopic to the identity or, equivalently,
 if $\sigma$ is in the 
kernel of the morphism $\Sigma_3 = \pi_1(B\Sigma_3) \to \pi_1(\Map(I^3,R), \beta^3)$.
\end{definition}

\begin{lemma}
If the ambient category $\bC$ is additive then every morphism $\beta: I \to R$ is cyclically invariant.  If $\bC$ is not necessarily additive but $\beta$ acts invertibly on $R$ then $\beta$ is cyclically invariant.
\end{lemma}
\begin{proof}
If $\bC$ is additive then $\Map(I^3,R)$ admits a grouplike $E_\infty$-structure, in particular the group $\pi_1(\Map(I^3,R), \beta^{3})$ is abelian. This forces the morphism
$\Sigma_3 \to \pi_1(\Map(I^3,R),\beta^3)$ to factor through the abelianization $(\Sigma_3)^{ab} \cong \Z/2$ which implies that $\sigma$ maps to zero. 

For the second case we claim that under the assumption that $\beta$ acts invertibly on $R$ it always follows that $\pi_1(\Map(I^3,R),¸\beta^3)$ is abelian and the same argument shows that $\sigma$ is cyclically invariant. 
We will in fact show that $\pi_1(\Map(I,R),\beta)$ is abelian for every  $\beta$ acting invertibly on $R$ and then apply this to the  element $\beta^3: I^3 \to R$. 
Tensoring $\beta$ with $I^{-1}$ we get a morphism $\beta': 1 \to R \otimes I^{-1}$ and the induced morphism $\pi_1(\Map(I,R),\beta) \to \pi_1(\Map(1,R \otimes I^{-1}, \beta')$ is an isomorphism. Thus we can reduce to 
the case of $\beta'$. In this case however the space $\Map(1, R \otimes I^{-1})$ gets an induced $E_\infty$-structure for which $\beta'$ admits an inverse. Thus it follows that 
$\pi_1(\Map(1, R \otimes I^{-1}),\beta')$ is abelian.
\end{proof}

Note that since $R$ is commutative, $\mu_\beta$ is a morphism of left $R$-modules (where the module structure on the target is 
defined using the fact that $\bC$ is symmetric monoidal). 
Assume $\beta$ is cyclically invariant. We define the localization of $M$ as 
$$
M[\beta^{-1}] := \colim\big( M \xrightarrow{\mu_\beta} I^{-1} \otimes M \xrightarrow{I^{-1} \otimes \mu_\beta} I^{-2} \otimes M  \to ... \big)
$$
where the colimit is taken in the category of (left) $R$-modules. This assignment $M \mapsto M[\beta^{-1}]$ refines to an endofunctor $L: \Mod(R) \to \Mod(R)$ with a natural transformation $\id \to L$. The following statement
is essentially due to Voevodsky and Robalo:
\begin{proposition}
$L$ is a localization and  the local objects are precisely the modules on 
which $\beta$ acts invertibly.
\end{proposition}
\begin{proof}
We first claim that $\beta$ acts invertibly on the module $M[\beta^{-1}]$. This can be seen unfolding the definitions and using the cyclic invariance condition as in the proof of  Proposition 4.21 in \cite{Robalo}. 
The second observation is that if $M$ has the property that $\beta$ acts invertibly, then the canonical morphism $M \to M[\beta^{-1}]$ is an equivalence (since then all the morphisms in the defining colimit are). Invoking 
\cite[Prop.~5.2.7.4]{HTT} these two statements together already imply the claim.
\end{proof}
\begin{corollary}
The $R$-module morphism $M\to M[\beta^{-1}]$ is universal among $R$-module morphisms $M\to N$ such that $\beta$ acts invertibly on $N$.
\end{corollary}

We now observe that the localization functor $L: \Mod(R) \to \Mod(R)$ can also be written as a tensor product of $R$-modules, namely
$$
M[\beta^{-1}] \simeq {R[\beta^{-1}] \otimes_{R} M.}
$$
This shows that $L$ is a smashing localization as discussed in \cite{Gepner:2013aa}. 
Using \cite[Lemma 3.6]{Gepner:2013aa}, we immediately get the following result.

\begin{proposition}
The module $R[\beta^{-1}]$ admits the unique structure of a commutative $R$-algebra such that the morphism $i: R \to R[\beta^{-1}]$ admits the structure of an $R$-algebra morphism. 
With this structure the morphism $i$ exhibits $R[\beta^{-1}]$
as the universal $R$-algebra on which $\beta$ acts invertibly.
\end{proposition}

Let us now outline the situations in which we need this construction.
\begin{example}
Let $\bC$ be the category of motivic spectra $\MotSp$ and $R$ a motivic commutative ring spectrum, i.e.~an element of $\CAlg(\MotSp)$. 
By construction, $\Sigma^{\infty}(\P^{1},\infty)$ is a tensor invertible object of $\MotSp$. Since $\MotSp$ is stable, it is also additive. Therefore, for any $\beta: \Sigma^\infty(\P^1,\infty)  \to R$
 we can form $R[\beta^{-1}]$ and it is the localization as described.
\end{example}
\begin{example}\label{ex:symmetric-monoidal-localization}
Let $\PR^{L}$ be the $\infty$-category of presentable $\infty$-categories and left adjoint functors. This $\infty$-category admits a symmetric monoidal structure in which the $\infty$-category of spaces 
$\Spc \simeq \sSet[W^{-1}]$ is the tensor unit. A commutative algebra object in $\PR^L$ is then exactly a presentably symmetric monoidal category $\bC$ \cite[Rem.~4.8.1.9.]{HA}.

By the universal property of the $\infty$-category of spaces -- it is freely generated under colimits by the one point space --, a
morphism $\Spc \to \bC$ in $\PR^L$ is essentially the same thing as an object of $\bC$. 
Thus every object $c \in \bC$ gives rise to such a morphism $\beta_c: \Spc \to \bC$. If $c$ is cyclically invariant (which for example follows if $c$ admits a cogroup structure), 
then  we can form the localization $\bC[\beta_c^{-1}]$ and it admits the structure of a presentably symmetric monoidal category.
This construction should not be confused with localizing a category at a set of morphisms.
\end{example}

\begin{remark}
We assumed that $\beta$ is cyclic invariant in $\bC$. This will be satisfied in all examples that we need. But for all our arguments we only need that $\beta$ acts `cyclically invariant' on the colimit $M[\beta^{-1}]$. 
In fact it is not hard to adapt our arguments from above to see that $\beta$ acts invertibly on $M[\beta^{-1}]$ precisely if it acts cyclically invariant. Thus all the results remain true if we only assume that $\beta$ acts cyclically invariant
on $R[\beta^{-1}]$. We do not know of an interesting example where this comes up. It is on the other hand not  hard to construct examples where $\beta$ does not act cyclically invariant and thus $R[\beta^{-1}]$ is not the 
localization.
\end{remark}

\bibliographystyle{amsalpha}
\bibliography{Hodge}

\providecommand{\bysame}{\leavevmode\hbox to3em{\hrulefill}\thinspace}
\providecommand{\MR}{\relax\ifhmode\unskip\space\fi MR }
\providecommand{\MRhref}[2]{%
  \href{http://www.ams.org/mathscinet-getitem?mr=#1}{#2}
}
\providecommand{\href}[2]{#2}
\begin{thebibliography}{GGN15}

\bibitem[Bei86]{BeilinsonHodge}
A.~A. Beilinson, \emph{Notes on absolute {H}odge cohomology}, Applications of
  algebraic {$K$}-theory to algebraic geometry and number theory, {P}art {I},
  {II} ({B}oulder, {C}olo., 1983), Contemp. Math., vol.~55, Amer. Math. Soc.,
  Providence, RI, 1986, pp.~35--68.

\bibitem[BG13]{bg}
U.~Bunke and D.~Gepner, \emph{Differential function spectra, the differential
  {B}ecker-{G}ottlieb transfer, and applications to differential algebraic
  {$K$}-theory}, \href{http://arxiv.org/abs/1306.0247}{arXiv:1306.0247}, 2013.

\bibitem[BGN18]{BarDualizing}
Clark Barwick, Saul Glasman, and Denis Nardin, \emph{Dualizing cartesian and
  cocartesian fibrations}, Theory Appl. Categ. \textbf{33} (2018), Paper No. 4,
  67--94.

\bibitem[BGT13]{GBT}
A.~J. Blumberg, D.~Gepner, and G.~Tabuada, \emph{A universal characterization
  of higher algebraic {$K$}-theory}, Geom. Topol. \textbf{17} (2013), no.~2,
  733--838.

\bibitem[BT15]{buta}
U.~Bunke and G.~Tamme, \emph{Regulators and cycle maps in higher-dimensional
  differential algebraic {$K$}-theory}, Adv. Math. \textbf{285} (2015),
  1853--1969.

\bibitem[BT16]{buta2}
\bysame, \emph{Multiplicative differential algebraic {$K$}-theory and
  applications}, Ann. K-Theory \textbf{1} (2016), no.~3, 227--258.

\bibitem[BW98]{Burgos-Wang}
J.I. Burgos and S.~Wang, \emph{Higher {B}ott-{C}hern forms and {B}eilinson's
  regulator}, Invent. Math. \textbf{132} (1998), no.~2, 261--305.

\bibitem[Day70]{Day}
B.J. Day, \emph{Construction of biclosed categories}, Ph.D. thesis, University
  of New South Wales, 1970.

\bibitem[Del71]{Deligne-HodgeII}
P.~Deligne, \emph{Th\'eorie de {H}odge. {II}}, Inst. Hautes \'Etudes Sci. Publ.
  Math. (1971), no.~40, 5--57.

\bibitem[Del74]{Deligne-HodgeIII}
\bysame, \emph{Th\'eorie de {H}odge. {III}}, Inst. Hautes \'Etudes Sci. Publ.
  Math. (1974), no.~44, 5--77.

\bibitem[DM15]{Deglise-Mazzari}
F.~D\'eglise and N.~Mazzari, \emph{The rigid syntomic ring spectrum}, J. Inst.
  Math. Jussieu \textbf{14} (2015), no.~4, 753--799.

\bibitem[Dre13]{Drew}
B.~Drew, \emph{R{\'e}alisations tannakiennes des motifs mixtes triangul{\'e}s},
  Ph.D. thesis, Universit{\'e} Paris 13, 2013.

\bibitem[Fel11]{Feliu:aa}
E.~Feliu, \emph{On uniqueness of characteristic classes}, J. Pure Appl. Algebra
  \textbf{215} (2011), no.~6, 1223--1242.

\bibitem[GGN15]{Gepner:2013aa}
D.~Gepner, M.~Groth, and T.~Nikolaus, \emph{Universality of multiplicative
  infinite loop space machines}, Algebr. Geom. Topol. \textbf{15} (2015),
  no.~6, 3107--3153.

\bibitem[Gil81]{GilletRR}
H.~Gillet, \emph{Riemann-{R}och theorems for higher algebraic {$K$}-theory},
  Adv. in Math. \textbf{40} (1981), no.~3, 203--289.

\bibitem[Gla16]{Glas}
S.~Glasman, \emph{Day convolution for {$\infty$}-categories}, Math. Res. Lett.
  \textbf{23} (2016), no.~5, 1369--1385.

\bibitem[GS09]{Gepner-Snaith}
D.~Gepner and V.~Snaith, \emph{On the motivic spectra representing algebraic
  cobordism and algebraic {$K$}-theory}, Doc. Math. \textbf{14} (2009),
  359--396.

\bibitem[HS15]{Holmstrom-Scholbach}
A.~Holmstrom and J.~Scholbach, \emph{Arakelov motivic cohomology {I}}, J.
  Algebraic Geom. \textbf{24} (2015), no.~4, 719--754.

\bibitem[Jar00]{Jardine}
J.~F. Jardine, \emph{Motivic symmetric spectra}, Doc. Math. \textbf{5} (2000),
  445--553 (electronic).

\bibitem[Joy]{Joyal}
A.~Joyal, \emph{The theory of quasi-categories and its applications}, Available
  at
  \href{http://mat.uab.cat/~kock/crm/hocat/advanced-course/Quadern45-2.pdf}{http://mat.uab.cat/~kock/crm/hocat/advanced-course/Quadern45-2.pdf}.

\bibitem[Lur09]{HTT}
J.~Lurie, \emph{Higher topos theory}, Annals of Mathematics Studies, vol. 170,
  Princeton University Press, Princeton, NJ, 2009.

\bibitem[Lur14]{HA}
\bysame, \emph{Higher algebra}, Book draft. Available from
  \href{http://www.math.harvard.edu/~lurie/}{www.math.harvard.edu/lurie}, 2014.

\bibitem[MG16]{Mazel-Gee}
A.~Mazel-Gee, \emph{Quillen adjunctions induce adjunctions of quasicategories},
  New York J. Math. \textbf{22} (2016), 57--93.

\bibitem[MV99]{Morel-Voevodsky}
F.~Morel and V.~Voevodsky, \emph{{${\bf A}^1$}-homotopy theory of schemes},
  Inst. Hautes \'Etudes Sci. Publ. Math. (1999), no.~90, 45--143 (2001).

\bibitem[NS17]{NikSag}
T.~Nikolaus and S.~Sagave, \emph{Presentably symmetric monoidal
  {$\infty$}-categories are represented by symmetric monoidal model
  categories}, Algebr. Geom. Topol. \textbf{17} (2017), no.~5, 3189--3212.

\bibitem[NS{\O}15]{Naumann:2010aa}
N.~Naumann, M.~Spitzweck, and P.A. {\O}stv{\ae}r, \emph{Existence and
  uniqueness of {$E_\infty$} structures on motivic {$K$}-theory spectra}, J.
  Homotopy Relat. Struct. \textbf{10} (2015), no.~3, 333--346.

\bibitem[PS08]{Peters-Steenbrink}
C.~Peters and J.~Steenbrink, \emph{Mixed {H}odge structures}, Ergebnisse der
  Mathematik und ihrer Grenzgebiete. 3. Folge. A Series of Modern Surveys in
  Mathematics [Results in Mathematics and Related Areas. 3rd Series. A Series
  of Modern Surveys in Mathematics], vol.~52, Springer-Verlag, Berlin, 2008.

\bibitem[Rob15]{Robalo}
M.~Robalo, \emph{{$K$}-theory and the bridge from motives to noncommutative
  motives}, Adv. Math. \textbf{269} (2015), 399--550.

\bibitem[S{\O}09]{Spitzweck-Ostvaer}
M.~Spitzweck and P.A. {\O}stv{\ae}r, \emph{The {B}ott inverted infinite
  projective space is homotopy algebraic {$K$}-theory}, Bull. Lond. Math. Soc.
  \textbf{41} (2009), no.~2, 281--292.

\bibitem[Wei94]{WeibelHomo}
C.~Weibel, \emph{An introduction to homological algebra}, Cambridge Studies in
  Advanced Mathematics, vol.~38, Cambridge University Press, Cambridge, 1994.

\end{thebibliography}

\end{document}